\documentclass[10pt]{amsart}
\usepackage[margin = 1.5in]{geometry}
\usepackage[utopia]{mathdesign}
\usepackage{epic,eepic,latexsym}

\usepackage{enumitem}
\usepackage{amsmath}
\usepackage{enumitem}
\usepackage{graphicx}
\usepackage{mathdots}
\usepackage{color}
\usepackage{diagbox}
\usepackage{array, makecell}
\usepackage{rotating}
\usepackage{braket}    
\usepackage{enumitem}
\usepackage{adjustbox}
\usepackage[all]{xy}

\usepackage{tikz}
    \usetikzlibrary{shapes.arrows}
    \usetikzlibrary{decorations.pathreplacing}
\usepackage{tikz-cd}
\usetikzlibrary{decorations}
\usepackage{url, hyperref}

\usepackage{amsthm}
\newtheorem{definition}{Definition}
\newtheorem{notation}[definition]{Notation}
\newtheorem{theorem}[definition]{Theorem}
\newtheorem*{theorem*}{Theorem}

\newtheorem{proposition}[definition]{Proposition}
\newtheorem*{proposition*}{Proposition}
\newtheorem{corollary}[definition]{Corollary}
\newtheorem{lemma}[definition]{Lemma}
\newtheorem{remark}[definition]{Remark}

\newtheorem{question}{Question}
\newtheorem*{question*}{Question}
\theoremstyle{definition}
\newtheorem*{notation*}{Notation}

\DeclareMathOperator{\Coll}{Coll}
\DeclareMathOperator{\Gr}{Gr}
\DeclareMathOperator{\Pic}{Pic}
\DeclareMathOperator{\spine}{sp}

\def\bC{\mathbb{C}}
\def\cC{\mathcal{C}}
\def\cL{\mathcal{L}}
\def\cO{\mathcal{O}}
\def\bP{\mathbb{P}}
\def\cV{\mathcal{V}}
\def\cW{\mathcal{W}}
\def\bZ{\mathbb{Z}}

\def\oM{\overline{\mathcal{M}}}
\def\cM{{\mathcal{M}}}
\def\omM{\overline{M}}
\def\C{\mathbb{C}}
\def\cH{\mathcal{H}}
\def\oH{\overline{\mathcal{H}}}

\newcommand{\Tev}{{\mathsf{Tev}}}

\title{Generalized Tevelev degrees of $\mathbb{P}^1$}

\author{Alessio Cela}

\address{ETH Z\"urich, Departement Mathematik,  R\"amisstrasse 101
\hfill \newline\texttt{}
 \indent 8044 Z\"urich, Switzerland} \email{{\tt alessio.cela@math.ethz.ch}}

\author{Carl Lian}

\address{Humboldt-Universit\"at zu Berlin, Institut f\"ur Mathematik,  Unter den Linden 6
\hfill \newline\texttt{}
 \indent 10099 Berlin, Germany} \email{{\tt liancarl@hu-berlin.de}}

\date{\today}

\begin{document}

\maketitle
\begin{abstract}
Let $(C,p_1,\ldots,p_n)$ be a general curve. We consider the problem of enumerating covers of the projective line by $C$ subject to incidence conditions at the marked points. These counts have been obtained by the first named author with Pandharipande and Schmitt via intersection theory on Hurwitz spaces and by the second named author with Farkas via limit linear series. In this paper, we build on these two approaches to generalize these counts to the situation where the covers are constrained to have arbitrary ramification profiles: that is, additional ramification conditions are imposed at the marked points, and some collections of marked points are constrained to have equal image. 
\end{abstract}

\tableofcontents

\section{Introduction}

Throughout, we work over $\bC$.

\subsection{Tevelev degrees}

Let $X$ be a smooth, projective variety. We are interested in the following question:

\begin{question}\label{tev_question}
Let $(C,p_1,\ldots,p_n)$ be a general pointed curve of genus $g$, let $\beta\in H_2(X,\bZ)$ be a curve class on $X$, and let $q_1,\ldots,q_n\in X$ be general points. Then, how many morphisms $f:C\to X$ are there in class $\beta$ (that is, $f_{*}([C])=\beta$) satisfying $f(p_i)=q_i$ for $i=1,2,\ldots,n$?
\end{question}

Equivalently, we ask for the set-theoretic degree of the canonical morphism
$$\tau:\cM_{g,n}(X,\beta)\to\cM_{g,n}\times X^n.$$
We assume on the one hand that the expected dimension of $\cM_{g,n}(X,\beta)$ is equal to the dimension of $\cM_{g,n}\times X^n$, equivalently
$$\int_{\beta}c_1(T_X)=\dim(X)(n+g-1),$$
and on the other hand that all dominating components of $\cM_{g,n}(X,\beta)$ are generically smooth of the expected dimension. In this case, the answer $\Tev^X_{g,\beta,n}$ is referred to as a \textbf{geometric Tevelev degree} of $X$. This question was considered by Tevelev in the case $X=\bP^1$, $\beta=(g+1)[\bP^1]$, $n=g+3$ in \cite{tevelev}.

Alternatively, one can formulate a virtual analogue of the question in Gromov-Witten theory in terms of the space of stable maps $\oM_{g,n}(X,\beta)$, equipped with its virtual fundamental class. The resulting counts are referred to as \textbf{virtual Tevelev degrees} and can be expressed in terms of the quantum cohomology of $X$, see \cite{bp} and \cite{Cela}. It is expected that when $X$ and $\beta$ are sufficiently positive, the geometric and virtual Tevelev degrees agree, see \cite{lp} for partial results in this direction.

We will deal in this work exclusively with geometric Tevelev degrees, so henceforth drop the modifier ``geometric.'' These are, in general, more difficult to compute than virtual Tevelev degrees, and at present complete answers are only available for $X=\bP^1$, which we review in the next section.

\subsection{Tevelev Degrees of $\bP^1$}
We now specialize to the case $X=\bP^1$. We will generalize the situation slightly, imposing the condition that $r\ge1$ of the marked points on the source curve $C$ map to the same marked point of the target. 

Fix a genus $g \geq 0$, an integer $\ell \in \mathbb{Z}$ and a positive integer $r \in \mathbb{Z}_{\geq 1}$. Call

\begin{equation}\label{old def of d and n}
d=g+1+\ell \hspace{0.15cm} \mathrm{and} \hspace{0.15cm} n=g+3+2 \ell.
\end{equation}

Assume 
\begin{equation}\label{old conditions on d and n}
r \leq d \hspace{0.15cm} \mathrm{and} \hspace{0.15cm} n-r+1 \geq 3.
\end{equation}

Let $\oH_{g,d,r}$ be the Deligne-Mumford stack parametrizing degree $d$ Harris-Mumford admissible covers
$$
\pi:(C,p_1,...,p_n) \rightarrow (D,q_1,...,q_{n-r+1})
$$
where $C$ is a genus $g$ nodal curve with $n$ distinct marked points $p_1,...,p_n \in C^{sm}$, $D$ is a genus $0$ nodal curve with $n-r+1$ distinct marked points $q_1,...,q_{n-r+1} \in D^{sm}$, $\pi$ sends $p_i$ to $q_i$ for $i=1,...,n-r$ and $p_i$ to $q_{n-r+1}$ for $ i \geq n-r+1$. We require $\pi$ to have exactly $2g+2d-2$ simple branch points $z_1,....,z_{2g+2d-2}$. In addition, we require that $\{ z_1,....,z_{2g+2d-2} \} \cap \{q_1,...,q_{n-r+1} \} = \emptyset$ and that $(D,q_1,...,q_{n-r+1},z_1,....,z_{2g+2d-2})$ is a stable curve. See Notation \ref{notaion for ehk=1} below for a more precise definition.

Define
$$
\tau_{g,\ell,r}: \oH_{g,d,r} \rightarrow \oM_{g,n} \times \omM_{0,n-r+1}
$$
to be the morphism that remembers only the stabilized domain and the target curves with all of the marked points (while ramification and branch points are forgotten and the curves are stabilized). Note that by condition \eqref{old def of d and n}, the domain and the target of $\tau_{g,\ell,r}$ have the same dimension. 

Finally, define the \textbf{Tevelev degree}
$$
\mathsf{Tev}_{g,\ell,r}= \frac{\mathrm{deg}(\tau_{g,\ell,r})}{(2g+2d-2)!}
$$
and set $\mathsf{Tev}_{g,\ell,r}=0$ in all other cases. In the case $r=1$, it is straightforward to check that this recovers the definition of the previous section, that is, $$\Tev_{g,\ell,1}=\Tev^{\bP^1}_{g,g+1+\ell,g+3+2\ell}=\Tev^{\bP^1}_{g,d,n}.$$ The factor $(2g+2d-2)!$ removes the redundancy of the different possible labelings of the ramification points on the domain curve.

In \cite{CPS}, explicit closed formulas for  $\mathsf{Tev}_{g,\ell,r}$ are found using the following recursion, obtained via intersection theory on the space $\oH_{g,d,n,r}$ after a nodal degeneration of $C$:

\begin{proposition*} \cite[Proposition 7]{CPS}
Fix $g, r \geq 1$ and  $\ell \in \mathbb{Z}$ satisfying the conditions in \eqref{old conditions on d and n}. Then,
we have the recursion
\begin{equation*} 
    \mathsf{Tev}_{g,\ell,r} = \mathsf{Tev}_{g-1,\ell,\max(1,r-1)} + \mathsf{Tev}_{g-1,\ell+1,r+1}\,.
\end{equation*}
\end{proposition*}

The genus $0$ case is instead established by hand and then the whole recursion is explicitly solved, giving the following:

\begin{theorem*} \cite[Theorem 12]{CPS}
Fix $g \geq 0$, $\ell \in \mathbb{Z}$ and $r \geq 1$ satisfying conditions \ref{old conditions on d and n}. Then, we have:

\vspace{5pt}
\noindent\ \ \ \ \ $\bullet$ for $r=1$,
\begin{equation*} 
    \mathsf{Tev}_{g,\ell,r} = 2^g -2 \sum_{i=0}^{-\ell-2} \binom{g}{i} + (-\ell  -2) \binom{g}{-\ell-1} + \ell \binom{g}{-\ell} \, ,
\end{equation*}
 
 \vspace{5pt}
\noindent \ \ \ \ \ $\bullet$ for $r>1$,
\begin{equation*} 
    \mathsf{Tev}_{g,\ell,r} = 2^g -2 \sum_{i=0}^{-\ell-2} \binom{g}{i} + (-\ell + r -3) \binom{g}{-\ell-1} + (\ell-1) \binom{g}{-\ell} - \sum_{i=-\ell+1}^{r-\ell-2} \binom{g}{i}\,.
\end{equation*}
\end{theorem*}

When $\ell\ge r-1$, all of the binomial coefficients are interpreted to vanish, leaving simply $\Tev_{g,\ell,r}=2^g$. This agrees with the virtual Tevelev degrees computed in \cite{bp}.

A different degeneration approach via the theory of limit linear series is given in \cite{FarkasLian}, resulting in the following formula:
\begin{theorem*}\cite[Theorem 1.3]{FarkasLian}
Fix $g \geq 0$, $\ell \in \mathbb{Z}$ and $r \geq 1$ satisfying conditions \eqref{old def of d and n} and \eqref{old conditions on d and n}. Then,
\begin{equation*} 
    \mathsf{Tev}_{g,\ell,r} = \int_{\Gr(2,d+1)}\sigma_1^g\sigma_{r-1}\cdot\left[\sum_{i+j=n-2-r}\sigma_{i}\sigma_j\right]-\int_{\Gr(2,d)}\sigma_1^g\sigma_{r-2}\cdot\left[\sum_{i+j=n-3-r}\sigma_{i}\sigma_j\right]
\end{equation*}
\end{theorem*}
Here, the second term is interpreted to vanish when $r=1$. In the case $(r,n)=(1,3)$, the formulas of \cite{CPS} and \cite{FarkasLian} recover the classical counts of Castelnuovo \cite{Cast}.

\subsection{Generalized Tevelev degrees for $\mathbb{P}^1$}\label{Generalized Tevelev degrees}
In this paper, we consider a natural generalization of the previous problem, in which incidence conditions with arbitrary ramification profiles are imposed. Let $g \geq 0$ be a genus, $\ell \in \mathbb{Z}$ and fix $k\ge0$ vectors of integers

$$\mu_h=(e_{h,1},....,e_{h,r_h}) \in \mathbb{Z}_{\geq 1}^{ r_h}$$ with $r_h \geq 1$ for $h=1,..,k$. For each $h$, define $$|\mu_h|=\sum_{j=1}^{r_h}e_{h,j}.$$

Call 
\begin{equation}\label{def of d}
d[g,\ell]=g+1+ \ell \hspace{0.15cm}
\end{equation}
and assume
\begin{equation}\label{dimensional constraint}
 g+3+2 \ell =\sum_{h=1}^{k} |\mu_h| 
\end{equation}
and that for every $h=1,...,k$ we have 
\begin{equation} \label{conditions on d}
|\mu_h| \leq d[g,\ell].
\end{equation}
%and that 
%\begin{equation}\label{condition on sum of r_h}
%\sum_{h=1}^k r_h \leq %n[g,\ell,\mu_1,...,\mu_k].
%\end{equation}
Also call
$$
n[\mu_1,...,\mu_k]= \sum_{h=1}^k r_h.
$$
the total number of markings and
$$
b[g,\ell,\mu_1,...,\mu_k]=2g+2d-2- \sum_{h=1}^k | \mu_h | + n
$$
the number of additional (simple) branch points.

\begin{notation}\label{notation_marked}
Unless there is some ambiguity, we will simply write $d,n$ and $b$ in place of $d[g,l]$, $n[\mu_1,...,\mu_k]$ and $b[g,l,\mu_1,...,\mu_k]$.
\end{notation}

\begin{definition}\label{def:hur_tev}
Denote by $\oH_{g,d,\mu_1,...,\mu_k}$ the Deligne-Mumford stack of admisssible covers (see also \cite{HM,acv}) whose objects over a scheme $S$ are commutative diagrams 
\[
\begin{tikzcd}
 C \arrow{rr}{\pi} \arrow{dr}  &&D \arrow{dl}\\
& S \arrow[lu, bend left=60, "p_i"] \arrow[ru, swap, bend right = 40, "q_j "]
  \arrow[ru, swap, bend right = 60, looseness=2, "Q_t"]
\end{tikzcd}
\]

where:

\begin{enumerate}[label=(\roman*)]
    \item $i \in \{1,...,n\}$, $j \in \{1,...,k \}$ and $t\in \{1,...,b\}$;
    \item $(C \rightarrow S,p_1,...,p_n)$ is a proper flat family of reduced connected genus $g$ curves with at most nodal singularities and $p_1,...,p_n:S \rightarrow C$ are sections lying in the smooth locus of $C/S$;
    \item $(D \rightarrow S,q_1,...,q_{k}, Q_1,...,Q_b)$ is a stable $(k+b)$-pointed curve of genus $0$;
    \item\label{incidence}  $\pi \circ p_i = q_{j}$ for $1 \leq j \leq k$ and $n- \sum_{h=j}^k r_h +1 \leq i \leq n- \sum_{h=j+1}^{k} r_h $ (see Figure 1 for an illustration); 
    \item $\pi^{-1}(\mathrm{Sing}(D/S) )= \mathrm{Sing( C/S)}$ set-theoretically;
    \item $\pi$ is finite and étale with the following exceptions:
    \begin{enumerate}
        \item for each $s \in S$ and each $x \in C_s$ a nodal point of $C_s$ (necessarily) sent to a nodal point $y$ of $D_s$ we require that locally $C,D$ and $\pi$ are described by:
        \begin{align*}
        & C: xy=a, \ a \in \widehat{\mathcal{O}_{S,s}}, \ x,y \ \mathrm{generate} \ \widehat{\mathfrak{m}_{C,x}} \\
        & D: uv=a^p, \ a \in \widehat{\mathcal{O}_{S,s}}, \hspace{0.3cm} u,v \ \mathrm{generate} \ \widehat{\mathfrak{m}_{D,y}}\\
        & \pi: u=x^p, v=y^p
        \end{align*}
        for some $p \in \mathbb{Z}_{\geq 1}$,
        \item there are unique smooth points $P_i:S \rightarrow C$, one over each $Q_i:S \rightarrow D$, where $\pi$ has ramification index $2$,
        \item  for $1 \leq h \leq k$ and $1 \leq j \leq r_h$, $\pi$ has ramification index $e_{h,j}$ at $p_{n-\sum_{s=h}^k r_s +j}$.
    \end{enumerate}
\end{enumerate}

Note that, by Riemann-Hurwitz formula, $\pi_s:C_s \rightarrow D_s$ is a degree $d$ map for all $s \in S$. The sections $p_i$ for $i=1,...,n$ and $q_i$ for $i=1,...,k$ are called \textbf{markings}, the sections $P_i$ for $i=1,...,b$ \textbf{ramification points} and the sections $Q_i$ for $i=1,...,b$ \textbf{branch points} of the cover $\pi$. 
\begin{figure}
\begin{center}
\begin{tikzpicture} [xscale=0.36,yscale=0.36]
	%\draw [help lines] (0,0) grid (30, 15);
			
	%line below
	%the q
	\draw [ultra thick, black] (5,0) to (34,0);

	\node at (11,0) {$\bullet$};
	\node at (11.7,-1) {$q_{1}$};
			
	\node at (18,-1) {$\dots$};
			
	\node at (20,0) {$\bullet$};
	\node at (21.5,-1) {$q_{k}$};
	
	%the Q
	\node[black!60!green] at (28,0) {$\bullet$};
	\node[black!60!green] at (28,-1) {$Q_1$};
			
	\node[black!60!green] at (30.8,-1) {$\dots$};
			
	\node[black!60!green] at (33,0) {$\bullet$};
	\node[black!60!green] at (33,-1) {$Q_b$};	
	%points above
	%the P
	\node[black!60!green] at (28,5) {$\bullet$};
	\node[black!60!green] at (28,4) {$P_1$};
			
	\node[black!60!green] at (30.8,4) {$\dots$};
			
	\node[black!60!green] at (33,5) {$\bullet$};
	\node[black!60!green] at (33,4) {$P_b$};
	
	%the first p
	
	%first box
	\node at (11,10) {$\bullet$};
	\node[blue] at (11.5,10.5) {$e_{1,1}$};
	\node at (11.7,9.4) {$p_{1}$};
			
	\node at (11,7.5) {$\vdots$};
			
	\node at (11,5) {$\bullet$};
	\node[blue] at (11.5,5.5) {$e_{1,r_1}$};
	\node at (11.7,4.4) {$p_{r_1}$};
			
	\draw[decorate,decoration={brace,amplitude=5pt,mirror}] (14,4.5) -- (14,10.5);
	\node at (15,7.5) {$r_1$};
	
	%last box
	\node at (20,12) {$\bullet$};
	\node[blue] at (20.5,12.5) {$e_{k,1}$};
	\node at (21.5,11) {$p_{n-r_k+1}$};
			
	\node at (20,8.5) {$\vdots$};
			
	\node at (20,5) {$\bullet$};
	\node[blue] at (20.5,5.5) {$e_{k,r_k}$};
	\node at (20.5,4) {$p_{n}$};
			
	\draw[decorate,decoration={brace,amplitude=5pt,mirror}] (24,4.5) -- (24,12.5);
	\node at (25,8.5) {$r_k$};
	
	%the arrow
	\draw[->] (18,3) to (18,1);
\end{tikzpicture}
\caption{Explanation of the indices.}
\end{center}
\end{figure}
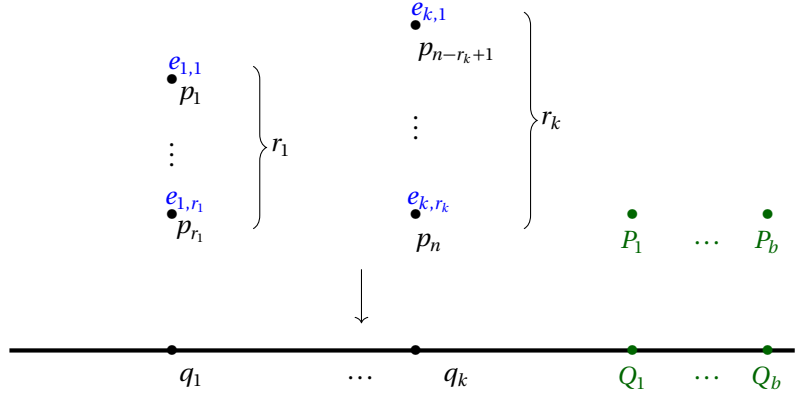

Morphisms in $\oH_{g,d,\mu_1,....,\mu_k}$ are cartesian diagrams
\[
\begin{tikzcd}
                                                &  & D \arrow[dd] \arrow[rrrr, "f_D"]                                                                                    &  &  &                                   & D' \arrow[dd]                                       \\
C \arrow[rrd] \arrow[rru, "\pi"] \arrow[rrrrr, "f_C"] &  &                                                                                                               &  &  & C' \arrow[rd] \arrow[ru, "\pi'"] &                                                                                                     \\
                                                &  & S \arrow[rrrr, "f_S"]  &  &  & & S' 
\end{tikzcd}
\]
with $f_C \circ p_i=f_S \circ p_i'$ for $i=1,...,n$, $f_D \circ q_j=f_S \circ q_j'$ for $j=1,...,k$ and $f_D \circ Q_t=f_S \circ Q_t''$ for $t=1,...,b$.
\end{definition}
Roughly speaking, these are Harris-Mumford admissible covers in \cite[Paragraph 4]{HM} with some extra data.
\begin{notation}\label{notaion for ehk=1}
When 
$$
\mu_h=(\underbrace{1,...,1}_{r_h})
$$
for all $h=1,...,k$, we simply write:
\begin{align*}
\oH_{g,d,r_1,...,r_k}&:=\oH_{g,d,\mu_1,...,\mu_k},\\
\tau_{g,\ell,r_1,...,r_1}&:=\tau_{g,\ell,\mu_1,...,\mu_k},\\
n[r_1,...,r_k]&:=n[\mu_1,...,\mu_k],\\
b[g,\ell,r_1,...,r_k]&:=b[g,\ell,\mu_1,...,\mu_k].
\end{align*}
% $\oH_{g,d,n,\mu_1,...,\mu_k}=\oH_{g,d,n,r_1,...,r_k}$, $\tau_{g,l,\mu_1,...,\mu_k}=\tau_{g,l,r_1,...,r_1},$ $n[g,l,\mu_1,...,\mu_k]=n[g,l,r_1,...,r_k]$ and $$.
\end{notation}
Assume now that
\begin{equation} \label{condition on k}
k \geq 3
\end{equation}
and let 
$$
\tau_{g,\ell,\mu_1,...,\mu_k}: \oH_{g,d,\mu_1,...,\mu_k} \rightarrow \oM_{g,n} \times \omM_{0,k}
$$
be the morphism that remembers only the domain and the target curves with all the marked points (while ramification and branch points are forgotten and the curves are stabilized).

Note that, by Equations \ref{def of d} and \ref{dimensional constraint}, the domain and the target of $\tau_{g,\ell,\mu_1,...,\mu_k}$ have both dimension 
$$
\mathrm{dim}[g,\ell,\mu_1,...,\mu_k]=4g+2\ell+(k-3)- \sum_{h=1}^k |\mu_h|+n=3g-3+n+(k-3)
$$

\begin{notation}
Unless there is ambiguity, we will write $\mathrm{dim}$ instead of $\mathrm{dim}[g,l,\mu_1,...,\mu_k]$.
\end{notation}

\begin{definition}
Define
$$
\mathsf{Tev}_{g,\ell,\mu_1,...,\mu_k}=
\begin{cases}
 \frac{\mathrm{deg}(\tau_{g,l,\mu_1,...,\mu_k})}{b!}&\text{when conditions \eqref{dimensional constraint}, \eqref{conditions on d} and \eqref{condition on k} are satisfied },\\
0  &\text{otherwise, }\, 
\end{cases} 
$$
to be the \textbf{generalized Tevelev degrees of $\mathbb{P}^1$}.
\end{definition}

The goal of this paper is to give explicit formulas for these Tevelev degrees following the two degeneration approaches of \cite{CPS} (via excess intersections on $\oH_{g,d,\mu_1,\ldots,\mu_{\ell}}$) and \cite{FarkasLian} (via limit linear series).

\subsection{Results and overview of the paper} 
In \S\ref{Reduction of the problem}-\ref{Explicit formulas}, we employ the strategy of \cite{CPS} to compute the generalized Tevelev degrees of $\bP^1$.

To begin, in \S\ref{Reduction of the problem}, we reduce the general problem to the case in which $e_{h,j}=1$ for all $h,j$. Namely, we prove the following equality:

\begin{proposition}\label{prop reduction}
Let $g \geq 0$, $\ell \in \mathbb{Z}$ and $\mu_1,...,\mu_k$ vectors with $\mu_h \in \mathbb{Z}_{\geq 1}^{ r_h}$ where $r_h \geq1$ for all $h=1,...,k$. Assume that the inequalities \eqref{dimensional constraint}, \eqref{conditions on d} and \eqref{condition on k} hold. Then we have:
$$
\mathsf{Tev}_{g,\ell,\mu_1,...,\mu_k}=\mathsf{Tev}_{g,\ell,|\mu_1|,...,|\mu_k|}.
$$

\end{proposition}

In \S\ref{Recursion via fiber geometry}, we use the same strategy as in \cite[Proof of Proposition 7]{CPS} to find a recursion for the Tevelev degrees: we degenerate $C$ to a 2-nodal curve obtained by attaching a rational bridge to a general smooth curve of genus $g-1$, and reduce, via excess intersections on Hurwitz spaces, to the case of genus $g=0$. We find:

\begin{proposition} \label{prop recursion}
Let $g \geq 1$, $\ell \in \mathbb{Z}$ an integer and $r_1,...,r_k \in \mathbb{Z}_{\geq 1}$ positive integers. Assume that the Equations \eqref{dimensional constraint}, \eqref{conditions on d}  and \eqref{condition on k} hold with 
$$
\mu_h=(\underbrace{1,...,1}_{r_h})
$$
for $h=1,...,k$. Then, we have the recursion 
$$
\mathsf{Tev}_{g,\ell,r_1,...,r_k}=\mathsf{Tev}_{g-1,\ell,r_1,...,r_{k-1},r_k-1}+\mathsf{Tev}_{g-1,\ell +1,r_1,...,r_{k-1},r_k+1},
$$
where by definition we set 
$$
\mathsf{Tev}_{g-1,\ell,r_1,...,r_{k-1},0}= \mathsf{Tev}_{g-1,\ell,r_1,...,r_{k-1}}
$$
\end{proposition}
The genus $0$ case is established by hand in \S\ref{The genus 0 case}. 

\begin{proposition} \label{prop for g=0}
For $\ell \geq 0$ and $r_1,...,r_k \in \mathbb{Z}_{\geq 1}$ positive integers such that the Equations \eqref{dimensional constraint}, \eqref{conditions on d}  and \eqref{condition on k} hold with 
$$
\mu_h=(\underbrace{1,...,1}_{r_h})
$$
for $h=1,...,k$. Then we have
$$
\mathsf{Tev}_{0,\ell,r_1,...,r_k}=1.
$$
\end{proposition}
Finally, in \S\ref{Explicit formulas}, we completely solve the recursion, obtaining the following closed formulas for the Tevelev degrees: 

\begin{theorem} \label{master formula}
Let $g \geq 0$, $\ell \in \mathbb{Z}$ and $\mu_1,...,\mu_k$ be vectors where $\mu_i \in \mathbb{Z}_{\geq 1}^{ r_i}$ with $r_i \geq1$. Assume that the Equations \eqref{dimensional constraint}, \eqref{conditions on d}  and \eqref{condition on k} hold. Then
\begin{align*}
\mathsf{Tev}_{g,l,\mu_1,...,\mu_k}= & \ 2^g -2\sum_{i=0}^{-\ell-2} {g \choose i} + \left(- \ell-k-2+\sum_{h=1}^k |\mu_h|\right){g \choose {-\ell-1}} \\
&+\left(\ell -k + \sum_{h=1}^k \delta_{|\mu_h|,1}\right) {g \choose {-\ell}}-\sum_{h=1}^k \sum_{i=- \ell+1}^{|\mu_h|- \ell -2} {g \choose i}.
\end{align*}
Here $\delta_{r_h,1}$ is the Kronecker delta.
\end{theorem}

In \S\ref{sec:schubert}, we give a second approach to the generalized Tevelev degrees, expanding on that of \cite{FarkasLian}. Namely, we consider a compact type degeneration of $C$ and apply the theory of Eisenbud-Harris limit linear series \cite{EH}. Some care is needed to handle incidence conditions imposed at marked points lying on different components of the degeneration of $C$. Our method essentially reduces the problem to a computation in genus 0, and yields the following:

\begin{theorem}\label{main_thm_schubert}
Let $g \geq 0$, $\ell \in \mathbb{Z}$ and $\mu_1,...,\mu_k$ be vectors where $\mu_h \in \mathbb{Z}_{\geq 1}^{ r_h}$ with $r_h \geq1$. Suppose that equations \eqref{dimensional constraint}, \eqref{conditions on d}  and \eqref{condition on k} hold. Then, the generalized Tevelev degree $\Tev_{g,\ell,\mu_1,\ldots,\mu_k}$ is equal to
\begin{align*}
    \sum_{(I,J):I\coprod J=\{1,2,\ldots,k\}}(-1)^{\# J}\cdot\int_{\Gr(2,d+1-\#J)}&\left(\prod_{h\in I}\sigma_{|\mu_h|-1}\prod_{h\in J}\sigma_{|\mu_h|-2}\right)\cdot\sigma_1^g\\
    \cdot&\left(\sum_{(i,j):i+j=k-3-\#J}\sigma_i\sigma_j\right).
\end{align*}
Here, we adopt the standard notation $\sigma_i=s_i(\cV)\in A^{i}(\Gr(2,d+1))$, where $\cV$ is the universal rank 2 subbundle on the Grassmannian $\Gr(2,d+1)$, so that $\sigma_{i}=0$ if $i<0$ or $i>d-1$.
\end{theorem}

In the case $|\mu_1|=\cdots=|\mu_k|=1$, the formula recovers that of \cite[Theorem 1.2]{FarkasLian}.

We also show in \S\ref{sec:recover_recursion} that the recursion of Proposition \ref{prop recursion} and the genus 0 base case of Proposition \ref{prop for g=0} can be extracted from the formula of Theorem \ref{main_thm_schubert}.

\subsection{Acknowledgments}

Much of this work was completed during the \textit{Helvetic Algebraic Geometry Seminar} in Geneva in August 2021, the second author's visit to ETH Z\"{u}rich in August 2021, and the first author's visit to HU Berlin in November 2021. We thank both institutions for their continued support, and Gavril Farkas, Rahul Pandharipande, and Johannes Schmitt for many discussions about Tevelev degrees. We also thank the anonymous referee for helpful comments on the first draft.

It is a pleasure for the first author to thank the second author for having introduced him to the theory of limit linear series and to thank once more Rahul Pandharipande for his support and for having suggested him the problem.

A.C. was supported by SNF-200020-182181. C.L. was supported by an NSF postdoctoral fellowship, grant DMS-2001976, and the MATH+ incubator grant ``Tevelev degrees.'' 

\section{Reduction to unramified marked points} \label{Reduction of the problem}
Let $g \geq 0$ be a genus, $\ell \in \mathbb{Z}$ an integer and $\mu_1,...,\mu_k$ vectors with $\mu_h \in \mathbb{Z}_{\geq 1}^{ r_h}$ where $r_h \geq1$ for all $h=1,...,k$. Assume that the Equations \eqref{dimensional constraint}, \eqref{conditions on d}  and \eqref{condition on k} hold.

In this section, we prove that the generalized Tevelev degrees depend only on the \textit{total} incidence order over each marked point of the target. Namely, we find the following equality, which implies Proposition \ref{prop reduction}:
\begin{equation}\label{restatement prop}
\mathsf{Tev}_{g,\ell,\mu_1,...,\mu_k}=\mathsf{Tev}_{g,\ell,\mu_1,...,\mu_{k-1},(|\mu_k|)}
\end{equation}
where $(|\mu_k|)$ is a vector in $\mathbb{R}$.

Note that we may assume $r_k \geq 2$. The strategy is simple: we show that, upon coalescing the points $p_{n-r_k+1},\ldots,p_n$ to into the single point $p_{n-r_k+1}$, the incidence conditions at $p_{n-r_k+1},\ldots,p_n$ turn into a single incidence condition at $p_{n-r_k+1}$ with ramification order $|\mu_k|$.

Formally, let $\Gamma_0$ be the stable graph
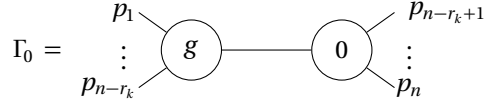
\[
\Gamma_0 \ = \ 
\begin{tikzpicture}[baseline = -0.1 cm]
%\draw [help lines] (0,0) grid (30, 15);
\draw (2,0) to (4,0);
\draw (2,0) -- (1.3,0.5);
\draw (2,0) -- (1.3,-0.5);

\draw (4,0) -- (4.7,0.5);
\draw (4,0) -- (4.7,-0.5);

\node at (1.1,0.5) {$p_1$};
\node at (0.9,-0.5) {$p_{n-r_k}$};
\node at (1.1,0) {$\vdots$};

\node at (5.4,0.5) {$p_{n-r_k+1}$};
\node at (4.9,-0.5) {$p_n$};
\node at (4.9,0) {$\vdots$};

\filldraw[white] (2,0) circle (0.4cm) node {$w_1$};
\filldraw[white] (4,0) circle (0.4cm) node {$w_2$};
\draw (2,0) circle (0.4cm) node {$g$};
\draw (4,0) circle (0.4cm) node {$0$};
\end{tikzpicture}
\]
corresponding to the boundary divisor of curves on $\oM_{g,n}$ with a rational tail containing $p_{n-r_k+1},\ldots,p_n$, and let $$\xi_{\Gamma_0}:\oM_{\Gamma_0}=\oM_{g,n-r_k+1}\times\omM_{0,r_k+1} \rightarrow \oM_{g,n}$$ be the usual gluing map. 

Following the notation of Definition \ref{def:hur_tev}, let $\overline{\mathcal{H}}_{0,|\mu_k|,\mu_k,(|\mu_k|)}$ be the Hurwitz space whose general point parametrizes covers $\pi:\bP^1\to\bP^1$ of degree $|\mu_k|$ sending $r_k$ marked points $p_{n-r_k+1},\ldots,p_n$ to a single point $q_{n-k}$ with ramification indices $e_{k,1},\ldots,e_{k,r_k}$, respectively, with an additional marked point $p$ of total ramification over the target. (We assume as usual that $\pi$ is otherwise simply branched.) Let $$
\varepsilon_{\mu_k}:\overline{\mathcal{H}}_{0,|\mu_k|,\mu_k,(|\mu_k|)}\to\omM_{0,r_k+1}
$$
be the map remembering the (stabilized) marked source.

The space $\overline{\mathcal{H}}_{0,|\mu_k|,\mu_k,(|\mu_k|)}$ also defines a boundary divisor $$\overline{\mathcal{H}}_{g,d,\mu_1,...,\mu_{k-1},(|\mu_k|)} \times \overline{\mathcal{H}}_{0,|\mu_k|,\mu_k,(|\mu_k|)}\to \overline{\mathcal{H}}_{g,d,\mu_1,...,\mu_k}$$ by gluing the source curves at the points of order $|\mu_k|$ ramification and gluing $d-|\mu_k|$ additional rational tails mapping isomorphically to the target, see also Lemma \ref{contributions in the reduction}.

We will show that there exists a commutative diagram
\begin{equation}
    \xymatrix{
    \displaystyle\coprod\overline{\mathcal{H}}_{g,d,\mu_1,...,\mu_{k-1},(|\mu_k|)} \times \overline{\mathcal{H}}_{0,|\mu_k|,\mu_k,(|\mu_k|)} \ar[r] \ar[d]_{\tau_{g,\ell,\mu_1,\ldots,\mu_{k-1},|\mu_k|}\times\varepsilon_{\mu_k}} & \overline{\mathcal{H}}_{g,d,\mu_1,...,\mu_k} \ar[d]^{\tau_{g,\ell,\mu_1,\ldots,\mu_k}} \\
    \oM_{g,n-r_k+1}\times\omM_{0,r_k+1}\times\omM_{0,k} \ar[r]^(.6){\xi_{\Gamma_0} \times 1} & \oM_{g,n}\times\omM_{0,k}}
\end{equation}
such that the induced map from the top left space to the cartesian product is birational onto the components of the fiber product that dominate $\oM_{g,n-r_k+k}\times\omM_{0,r_k+1}\times\omM_{0,n-r_k+k}$. Here, the disjoint union is over a certain set of combinatorial choices, as detailed in Lemma \ref{contributions in the reduction} and the two maps 
\begin{align*}
    \tau_{g,\ell,\mu_1,\ldots,\mu_k}&:\overline{\mathcal{H}}_{g,d,\mu_1,...,\mu_k}\to\oM_{g,n}\times\omM_{0,k}\\
    \tau_{g,\ell,\mu_1,\ldots,\mu_{k-1},|\mu_k|}&:\overline{\mathcal{H}}_{g,d,\mu_1,...,\mu_{k-1},(|\mu_k|)}\to\oM_{g,n-r_k+1}\times\omM_{0,k}
\end{align*}
are as in \S\ref{Generalized Tevelev degrees}. We will conclude by comparing their degrees.

We now proceed to the proof. We have a commutative diagram 

% let $$(D,q_1,...,q_{n-r_{\tot}+k}) \in M_{0,n-r_{\tot}+k}$$
% be a fixed general pointed genus $0$ curve, and let $$((C,p_1,...,p_{n-r_k},x),(P,y,p_{n-r_k+1},...,p_n)) \in \cM_{\Gamma_0}$$
% be a fixed general point. Here, $x \in C$ and $y \in P$ are the points at which $C$ and $P$ are glued by the usual gluing map $\xi_{\Gamma_0}:\oM_{\Gamma_0} \rightarrow \oM_{g,n}$.

\begin{equation} \label{eqn:Cfibrediagram}
\begin{tikzcd}
 &\displaystyle\coprod_{(\Gamma,\Gamma')} \overline{\mathcal{H}}_{(\Gamma, \Gamma')} \arrow[d] \arrow[r,"\coprod \xi_{(\Gamma,\Gamma')}"]
 & \overline{\mathcal{H}}_{g,d,\mu_1,...,\mu_k}\arrow[d]\\
 &\displaystyle\coprod_{\widehat {\Gamma_0}} \oM_{\widehat \Gamma_0} \times \omM_{0,k} \arrow[d] \arrow[r,"\coprod \xi_{\widehat{ \Gamma}_0} \times 1"] & \oM_{g,n + b}\times \omM_{0,k} \arrow[d]\\
 & \oM_{\Gamma_0}\times \omM_{0,k} \arrow[r,"\xi_{\Gamma_0} \times 1"]& \oM_{g,n} \times \omM_{0,k}
\end{tikzcd}
\end{equation}

In the middle row, we have added back the data of the $b$ simple ramification points, so that the composite vertical arrow on the right is again $\tau_{g,\ell,\mu_1,...,\mu_k}$. 
% Here, the upper right vertical arrow 
% is the map remembering the domain together with the $n$ markings and $b$ simple ramification points and the target together with $n-r_{\tot}+k$ markings. The bottom right vertical arrow
% is the map forgetting $b$ marked points on the genus $g$ curve, so that the composition on the right is the familiar map $\tau_{g,l,\mu_1,...,\mu_k}$.
The lower disjoint union on the left side of the diagram is over all stable graphs $\widehat{\Gamma}_0 $ that can be obtained by distributing $b$ legs to the two vertices of $\Gamma_0$. The lower square in the diagram is then not Cartesian, but the disjoint union of the $\oM_{\widehat \Gamma_0} \times \omM_{0,k}$ maps properly and birationally to the corresponding fiber product.

The upper left entry $\coprod\overline{\mathcal{H}}_{(\Gamma, \Gamma')}$ is the disjoint union of boundary divisors of $\overline{\mathcal{H}}_{g,d,\mu_1,...,\mu_k}$, where the index set of the disjoint union consists of
isomorphism classes of $\widehat \Gamma_0$-structures on $\Gamma$ (See \cite[Definition 2.5]{SvZ} ) such that the composition $$E(\widehat \Gamma_0 ) \subseteq E(\Gamma) \rightarrow E(\Gamma')$$ is surjective. Then, by \cite[Proposition 3.2]{Lian}, the upper square is Cartesian (a priori) \textit{on the level of closed points}, {and $\overline{\mathcal{H}}_{(\Gamma, \Gamma')}$ is identified with the underlying reduced space of the fiber product. (Note: our diagram differs from that of \cite[Proposition 3.2]{Lian} in that we have extra factors of $\omM_{0,k}$ parametrizing the marked target, on which the horizontal maps act by the identity; the Cartesianness of the upper square is preserved.)

Finally, we denote the composition
$$
\overline{\mathcal{H}}_{(\Gamma, \Gamma')} \rightarrow \oM_{\Gamma_0}\times \omM_{0,k}.
$$
by $\tau_{(\Gamma,\Gamma')}$.

% We have
% \begin{equation} \label{how we compute the degree}
% \mathrm{deg}(\tau_{g,l,\mu_1,...,\mu_k})=\mathrm{deg} \Big( \{ (C,P,D) \}. \big( \oM_{\Gamma_0}\times \oM_{0,n-r_{\tot}+k} \big) . \overline{\mathcal{H}}_{g,d,n,\mu_1,...,\mu_k} \Big)
% \end{equation}
% where the point " . " stands for the refined intersection product \textcolor{red}{this looks to me like strange notation, why not $\cdot$?} and
% \begin{multline} \label{first intersection}
% \big( \oM_{\Gamma_0}\times \oM_{0,n-r_{\tot}+k} \big) . \overline{\mathcal{H}}_{g,d ,n,\mu_1,...,\mu_k} =  \sum_{(\Gamma,\Gamma')} \left\{ c(N_{(\Gamma,\Gamma')}) \cap s(\overline{\mathcal{H}}_{(\Gamma, \Gamma')},\overline{\mathcal{H}}_{g,d,n,\mu_1,...,\mu_k}) \right\}_{\mathrm{dim}-1}
% \end{multline}
% where the sum is over all the pairs $(\Gamma,\Gamma')$ appearing in diagram \eqref{eqn:Cfibrediagram}, the line bundle $N_{(\Gamma,\Gamma')}=\tau_{(\Gamma,\Gamma')}^* N_{\xi_{\Gamma_0} \times 1}$ is the pullback under $\tau_{(\Gamma,\Gamma')}$ of the normal bundle of $\xi_{\Gamma_0} \times 1$, the class $s(\overline{\mathcal{H}}_{(\Gamma, \Gamma')},\overline{\mathcal{H}}_{g,d,n,\mu_1,...,\mu_k})$ is the Segre class of $\overline{\mathcal{H}}_{(\Gamma, \Gamma')}$ in $\overline{\mathcal{H}}_{g,d,n,\mu_1,...,\mu_k}$ and the terms in the sum are all in degree $\mathrm{dim}-1$.

Because we are interested in degrees of the vertical maps, we need only consider components $\overline{\mathcal{H}}_{(\Gamma, \Gamma')}$ mapping dominantly under $\tau_{(\Gamma,\Gamma')}$. We will see in Corollary \ref{cartesianity of our diagram} that, when restricted to such components, the upper part of diagram \eqref{eqn:Cfibrediagram} is in fact Cartesian. The following lemma shows that, up to relabelling of the ramification and branch legs, there is only one possible $(\Gamma,\Gamma')$ for which $\tau_{(\Gamma,\Gamma')}$ is dominant.

\begin{lemma}\label{contributions in the reduction}
The pairs $(\Gamma,\Gamma')$ yielding dominating components in Diagram \eqref{eqn:Cfibrediagram} 
%and contributing in Equation \eqref{how we compute the degree} 
have the following form:
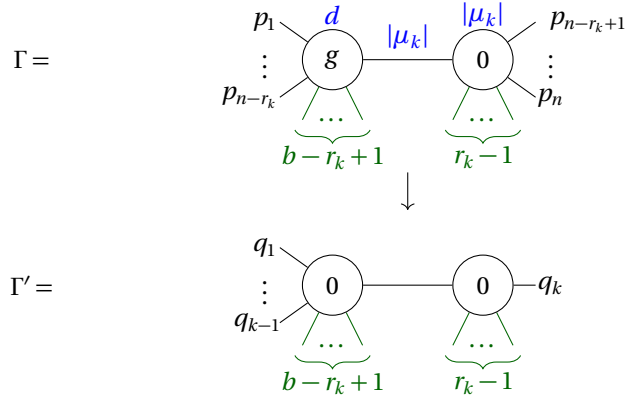
\[
\begin{tikzpicture}[baseline = -0.1 cm]
%\draw [help lines] (-5,-5) grid (5, 5);

%Here Gamma starts
\node at (-2,0) {$\Gamma=$};
\draw (2,0) to (4,0);
\node[blue] at (3,0.3) {$|\mu_k|$};

%black legs on g

\draw (2,0) -- (1.3,0.5);
\draw (2,0) -- (1.3,-0.5);

\node at (5.4,0.5) {$p_{n-r_k+1}$};
\node at (4.9,-0.5) {$p_n$};
\node at (4.9,0) {$\vdots$};

%green legs on g

\draw[black!60!green] (2,0) -- (1.6,-0.8);
\node[black!60!green] at (2,-0.8) {$\dots$};
\draw[black!60!green] (2,0) -- (2.4,-0.8);

\draw [black!60!green, decorate,decoration={brace,amplitude=5pt,mirror,raise=1pt}] (1.5,-0.9) -- (2.5,-0.9) ;
\node[black!60!green] at (2,-1.3) {$b-r_k+1$};

%black legs on 0 in Gamma

\draw (4,0) -- (4.7,0.5);
\draw (4,0) -- (4.7,-0.5);

\node at (1.1,0.5) {$p_1$};
\node at (0.9,-0.5) {$p_{n-r_k}$};
\node at (1.1,0) {$\vdots$};

%green legs on 0 in Gamma

\draw[black!60!green] (4,0) -- (3.6,-0.8);
\node[black!60!green] at (4,-0.8) {$\dots$};
\draw[black!60!green] (4,0) -- (4.4,-0.8);

\draw [black!60!green, decorate,decoration={brace,amplitude=5pt,mirror,raise=1pt}] (3.5,-0.9) -- (4.5,-0.9);
\node[black!60!green] at (4,-1.3) {$r_k-1$};

%vertices in Gamma

\filldraw[white] (2,0) circle (0.4cm) node {$w_1$};
\filldraw[white] (4,0) circle (0.4cm) node {$w_2$};

\draw (2,0) circle (0.4cm) node {$g$};
\draw (4,0) circle (0.4cm) node {$0$};

\node[blue] at (2,0.6) {$d$};
\node[blue] at (4,0.6) {$|\mu_k|$};

%arrow down

\draw[->] (3,-1.5) -- (3,-2);

%Here Gamma' starts
\node at (-2,-3) {$\Gamma'=$};
\draw (2,-3) to (4,-3);

%black legs on the left vertex of Gamma'
\draw (2,-3) -- (1.3,-2.5);
\draw (2,-3) -- (1.3,-3.5);

\node at (1.1,-2.5) {$q_1$};
\node at (1,-3.5) {$q_{k-1}$};
\node at (1.1,-3) {$\vdots$};

%green legs on the left vertex of Gamma'

\draw[black!60!green] (2,-3) -- (1.6,-3.8);
\node[black!60!green] at (2,-3.8) {$\dots$};
\draw[black!60!green] (2,-3) -- (2.4,-3.8);

\draw [black!60!green, decorate,decoration={brace,amplitude=5pt,mirror,raise=1pt}] (1.5,-3.9) -- (2.5,-3.9) ;
\node[black!60!green] at (2,-4.3) {$b-r_k+1$};

% black leg on the right vertex of Gamma'
\draw (4,-3) -- (4.7,-3);

\node at (4.9,-3) {$q_{k}$};

% green legs on the right vertex of Gamma'
\draw[black!60!green] (4,-3) -- (3.6,-3.8);
\node[black!60!green] at (4,-3.8) {$\dots$};
\draw[black!60!green] (4,-3) -- (4.4,-3.8);

\draw [black!60!green, decorate,decoration={brace,amplitude=5pt,mirror,raise=1pt}] (3.5,-3.9) -- (4.5,-3.9) ;
\node[black!60!green] at (4,-4.3) {$r_k-1$};

%vertices of Gamma'

\filldraw[white] (2,-3) circle (0.4cm) node {$w_1$};
\filldraw[white] (4,-3) circle (0.4cm) node {$w_2$};
\draw (2,-3) circle (0.4cm) node {$0$};
\draw (4,-3) circle (0.4cm) node {$0$};

\end{tikzpicture}
\]

Here, the green legs correspond to ramification and branch points, the black legs correspond to markings, the blue numbers over a vertex of  $\Gamma$ indicate the degree of the map when restricted to the corresponding vertex and the blue numbers over an edge indicate the ramification index at the correspond node. In particular, over the right vertex of $\Gamma'$, there are $d-|\mu_k|$ additional genus $0$ vertices of $\Gamma$, each with degree $1$, which are not drawn. Also, the picture does not depict the fact that our covers are subject to restrictions (iii) and (iv.c) appearing in the definition of $\overline{\mathcal{H}}_{g,d,\mu_1,....,\mu_k}$.
\end{lemma}
\begin{proof}
We start with some preliminary observations:
\begin{enumerate}[label=(\roman*)]
    %\item In order for $(\Gamma,\Gamma')$ to contribute in Equation \eqref{eqn:Cfibrediagram}, the map $\tau_{(\Gamma,\Gamma')}$ must be surjective. 
    \item In order for the map $\tau_{(\Gamma,\Gamma')}$ to be surjective, the graph obtained from $\Gamma$ forgetting all the legs corresponding to ramification points and then stabilizing is canonically identified with $\Gamma_0$. In particular, the set $V(\Gamma)$ contains canonically defined vertices $v_g$ and $v_0$ of genus $g$ and $0$, respectively, corresponding to the vertices of $\Gamma_0$, which persist after the contraction. We also denote the degrees of $v_g$ and $v_0$ by $d_g$ and $d_0$, respectively.
    \item From the fact that the graph $\Gamma$ comes endowed with a $\widehat \Gamma_0$ structure such that the composition $E(\widehat \Gamma_0 ) \subseteq E(\Gamma) \rightarrow E(\Gamma')$ is surjective, we have that $\Gamma'$ has exactly one edge and two vertices. Denote by $w_0$ and $w_1$ the vertices of $\Gamma'$. %In particular, we see again that $\mathrm{dim}(\oH_{(\Gamma,\Gamma')})= \mathrm{dim}-1$;
    \item Every vertex $v$ of $\Gamma$ different from $v_0$ and $v_g$ cannot contain more than one ramification or marked point. Indeed, if $v$ contains multiple marked points, then the image of $\tau_{(\Gamma,\Gamma')}$ would be supported inside $$\partial(\oM_{\Gamma_0})\times\omM_{0,k}$$ where $\partial$ denotes the boundary. On the other hand, if $v$ contains a ramification point in addition to another ramification or marked point, then the fibers of $\tau_{(\Gamma,\Gamma')}$ must have positive dimension, because the positions of the ramification points of a cover do not influence its image under $\tau_{(\Gamma,\Gamma')}$. In either case, we have a contradiction of the surjectivity of $\tau_{(\Gamma,\Gamma')}$.
    %\item (ii) implies that $v_g$ and $v_0$ are joined by an edge $e$ of $\Gamma$. Let $a$ be the ramification index at the corresponding node. Moreover, and all other edges of $\Gamma$ are leaves, incident to one of $v_0$ or $v_g$ and no other vertices. 
    % In particular, we have
    % \begin{align*} \label{dim H(Gamma,Gamma')}
    % \mathrm{dim}(\oH_{(\Gamma,\Gamma')})= & \ b+(n-r_{\tot} +k) -4
    % =  \ \mathrm{dim}(\oM_{ \Gamma_0} \times \oM_{0,n-r_{\tot}+k}).
    % \end{align*}
    % \textcolor{red}{isn't it automatic that $\oH_{(\Gamma,\Gamma')}$ is a boundary divisor because it's the pullback of a boundary divisor?}
    \item In order for the map $\tau_{(\Gamma,\Gamma')}$ to be surjective, exactly one vertex of $\Gamma'$ contains at most one leg corresponding to a marked point of the target. Without loss of generality, take $w_1$ to be this vertex.
    %\item If $\Gamma_g$ denotes the subgraph of $\Gamma$ corresponding to the genus $g$ vertex in $\widehat \Gamma_0$ (see \cite[Definition 2.2(v)]{SvZ}), then for each marked point $p_1,...,p_{n-r_k}$, there is a unique path in $\Gamma_g$ from $v_g$ to the corresponding leg. Moreover, by the surjectivity of $\tau_{(\Gamma,\Gamma')}$, any two of these paths only meet at $v_g$. A similar observation holds for $v_0$ and $p_{n-r_k+1},...,p_n$.
\end{enumerate}

Note (iii) implies that all edges of $\Gamma$ different from $v_0$ and $v_g$ are leaves, incident to one of $v_0$ or $v_g$ and no other vertices. In particular, $v_g$ and $v_0$ are joined by an edge $e$ of $\Gamma$ corresponding to a node whose ramification index we denote $a\in\mathbb{Z}_{\geq 1}$. Moreover, we have that $v_g$ and all of the leaves attached to $v_0$ map to one of $w_0,w_1$, and that $v_0$ and all of the leaves attached to $v_g$ map to the other vertex of $\Gamma'$.

We distinguish two cases. 

\vspace{10pt}
\noindent \textbf{Case I)} $v_g$ maps to $w_0$. \vspace{10pt}

Let $\oH_{w_0},\oH_{w_1}$ be the factors of $\oH_{(\Gamma,\Gamma')}$ corresponding to the cover over $w_0,w_1$, respectively. Denoting by $n_0,n_1$ be the valences of $w_0,w_1$, respectively, we have
\begin{align*}
    \dim(\oH_{w_0})&=n_0-3,\\
    \dim(\oH_{w_1})&=n_1-3.
\end{align*}

We will show via a parameter count that the only way for the map
$$
\tau_{(\Gamma,\Gamma')}: \oH_{(\Gamma,\Gamma')}=\oH_{w_0}\times\oH_{w_1} \to (\oM_{g,n-r_k+1} \times \omM_{0,k})\times \omM_{0,r_k+1}
$$
to be surjective is for $(\Gamma,\Gamma')$ to have the claimed form. Observe that the composition of $\tau_{(\Gamma,\Gamma')}$ with the projection to $\oM_{g,n-r_k+1} \times \omM_{0,k}$ does not depend on the factor $\oH_{w_1}$. Similarly, the composition of $\tau_{(\Gamma,\Gamma')}$ with the projection to $\omM_{0,r_k+1}$ does not depend on the factor $\oH_{w_0}$. 

It follows that
\begin{align*}
    \dim(\oH_{w_0})&\ge \dim(\oM_{g,n-r_k+1} \times \omM_{0,k})\\
    \dim(\oH_{w_1})&\ge \dim(\omM_{0,r_k+1}),
\end{align*}
and then that both inequalities are in fact equalities because the source and target of $\tau_{(\Gamma,\Gamma')}$ each have dimension $\dim-1$. In particular, we have $n(w_1)=r_k+1$.
% By Observations (iv) and (v), the factor of $\oH_{(\Gamma,\Gamma')}$ over $w_0$ completely determines the map...

% Moreover, it forms a finite cover of $\oM_{0,n(w_0)}$, so we have

% This means that at most $r_k$ between the marked and the ramification points lie on $w_1$. 
% \\Note also that, by Observation (iv), $v_0$ must instead lie over $w_1$ and thus, by Observation (vi), we have
% \begin{equation}\label{n(w1)}
% n(w_1)-3 = \mathrm{dim}(\oH_{w_1}) \geq r_k-2.
% \end{equation}
Now, we claim that the legs corresponding to $p_{n-r_k+1},...,p_n$ are incident to the vertex $v_0$, and that the only edge attached to $v_0$ is $e$. Suppose instead that these legs lie on genus 0 vertices (distinct from $v_g$) attached to $v_0$ and mapping to $w_0$. Note that these vertices cannot contain legs corresponding ramification point, by (iii) above. Therefore, we have $r_k$ distinct vertices of genus 0 containing the legs corresponding to $p_{n-r_k+1},...,p_n$. Let also $s$ be the number of genus $0$ vertices (also distinct from $v_g$) mapped to $w_0$ with degree $2$, joined to $v_0$ by an edge with ramification index $2$ and carrying exactly one leg of $\Gamma$ corresponding to a ramification point. 

Then, the degree of $v_0$ would be at least
$$
d_0=a+ 2s + |\mu_k| 
$$
and the number of legs on $v_0$ corresponding to ramification points would be at least
\begin{equation*}
(2d_0-2)-(a-1)-(|\mu_k|-r_k) - s =a+3s+|\mu_k|-1+r_k,
\end{equation*}
which is strictly larger than the number $n(w_1)-1=r_k$ allowed. This proves the claim. 

We deduce that $d_g=d$ and $a=d_0 \geq |\mu_k|$. Finally, using again the fact that $n(w_1)=r_k+1$, we find that in fact $a=d_0=|\mu_k|$ and that there are no ramification points other than those attached to $v_0$ over $w_1$. It follows that $\Gamma \rightarrow \Gamma'$ must be of the claimed form.

\vspace{10pt}
\noindent \textbf{Case II)} $v_g$ maps to $w_1$. \vspace{10pt}

We show that this is not possible. Let $b$ be the total number of additional simple ramification points as before. Arguing as in the previous case, we find that $\dim(\oH_{w_1})=n(w_1)-3=b-r_k+1$.
% and
% $$
% n(w_0)-3=n - \sum_{h+1}^{k-1}+k-5.
% $$
We claim that the maximal number of legs corresponding to ramification points and lying over $w_1$ is $b-r_k+1$. This yields a contradiction, as in this case $\mathrm{dim}(\oH_{w_1}) \leq b-r_k+1+2-3=b-r_k$. 

To prove the claim, suppose that over $w_1$ there are exactly $s$ additional genus $0$ vertices with degrees equal to $2$ and containing each exactly $1$ simple ramification point. Also let $\gamma$ be equal to $|\mu_k|$ if the marked point on $w_1$ is $q_{k}$ and $0$ otherwise. Then the maximum number of ramification points over $w_1$ is at most
\begin{align*}
&\textcolor{white}{\leq}(2 d_g +2g-2) - (a-1) +s -\sum_{h=1}^{k-1}|\mu_h| + \sum_{h=1}^{k-1}r_h  \\
 &\leq 2d+2g-2-(a-1)-3s -2 \gamma -\sum_{h=1}^{k-1}|\mu_h| + \sum_{h=1}^{k-1}r_h \\
 &\leq 2d+2g-2+1-s-\gamma-|\mu_k| -\sum_{h=1}^{k-1}|\mu_h| + \sum_{h=1}^{k-1}r_h\\
 &=b-r_k+1-s-\gamma,
\end{align*}
where in the first inequality we use the fact that $d_g+2s +\gamma \leq d$ and in the second inequality the fact that $a+2s+ \gamma \geq |\mu_k|$.
\end{proof}

% \begin{lemma}\label{contributions_reduced}
% Let $(\Gamma,\Gamma')$ be as in Lemma \ref{contributions in the reduction}. Then, the corresponding component of the fiber product in \eqref{eqn:Cfibrediagram} to which $\oH_{(\Gamma,\Gamma')}$ maps is reduced.
% \end{lemma}

% \begin{proof}
% This follows from a computation on the level of complete local rings and the fact that the covers parametrized by $\oH_{(\Gamma,\Gamma')}$ have exactly one ramified node, see \cite[\S 3.2]{Lian}.
% \end{proof} 

\begin{corollary} \label{cartesianity of our diagram}
The upper square in diagram \eqref{eqn:Cfibrediagram} is Cartesian when restricted to the components $\oH_{(\Gamma,\Gamma')}$ of Lemma \ref{contributions in the reduction}.
\end{corollary}
\begin{proof}
We already have that the square in question from \eqref{eqn:Cfibrediagram} is Cartesian on the level of reduced spaces. On the other hand, a computation on the level of complete local rings shows that the fiber product in the upper square is reduced when restricted to the components $\oH_{(\Gamma,\Gamma')}$ because the covers parametrized by $\oH_{(\Gamma,\Gamma')}$ have exactly one ramified node, see \cite[\S 3.2]{Lian}.
\end{proof}

\begin{lemma} \label{degree varpsilon mu}
Fix $\mu=(e_1,...,e_r) \in \mathbb{Z}_{\geq 1}^r$ and let $\varepsilon_\mu:\oH_{0,|\mu|,(|\mu|),\mu} \rightarrow \omM_{0,r+1}$ be the morphism recalling the (stabilized) domain genus $0$ curve with all the markings, as defined above. Then,
$$
\mathrm{deg}(\varepsilon_{\mu})=(r-1)!.
$$
\end{lemma}
\begin{proof}
Up to scaling, there is a unique meromorphic function with zeroes and poles prescribed by the $r+1$ marked points. Then, $(r-1)!$ is the number of ways the ways to label the ramification points of the resulting cover.
\end{proof}

\begin{proof}[Proof of Proposition \ref{prop reduction}]
We compare the Tevelev degrees arising from diagram \eqref{eqn:Cfibrediagram}. We have
\begin{align*}
\mathrm{deg}(\tau_{g,l,\mu_1,...,\mu_r})&= {b[g,l,\mu_1,...,\mu_r] \choose r_k-1}(r_k-1)! \ \mathrm{deg}(\tau_{g,l,\mu_1,...,\mu_{r-1},(|\mu_r|)}) \\
&= b[g,l,\mu_1,...,\mu_k]! \ \mathsf{Tev}_{g,l,\mu_1,...,\mu_{k-1},(|\mu_{k}|)}
\end{align*}
where the first factor counts the number of possible $(\Gamma,\Gamma')$ as in Lemma \ref{contributions in the reduction} and the second factor comes from Lemma \ref{degree varpsilon mu}. The multiplicities of the contributions from each $\oH_{(\Gamma,\Gamma')}$ are 1, owing to Corollary \ref{cartesianity of our diagram}. Equation \eqref{restatement prop} follows, so we are done.
\end{proof}

\section{Genus recursion} \label{Recursion via fiber geometry}
In this section, we prove the recursion of Proposition \ref{prop recursion}. The idea is exactly the same as that used in \cite[Proof of Proposition 7]{CPS}, which we briefly recall here for the reader's convenience.

Let $g \geq 1$ be a genus, $\ell \in \mathbb{Z}$ an integer and $r_1,...,r_k \in \mathbb{Z}_{\geq 1}$ positive integers. Set 
$$
\mu_h=(\underbrace{1,...,1}_{r_h})
$$
for $h=1,...,k$, and assume that the Equations \eqref{dimensional constraint}, \eqref{conditions on d}  and \eqref{condition on k} hold. Let $\Gamma_0$ be the graph
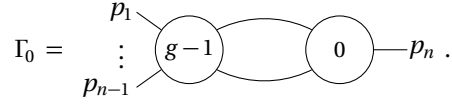
\[
\Gamma_0 \ = \ 
\begin{tikzpicture}[baseline = -0.1 cm]
%\draw [help lines] (-5,-5) grid (5, 5);
\draw (2,0) to [out=45,in=135] (4,0);
\draw (2,0) to [out=-45,in=-135] (4,0);

\draw (2,0) -- (1.3,0.5);
\draw (2,0) -- (1.3,-0.5);

\draw (4,0) -- (4.9,0);

\node at (1.1,0.5) {$p_1$};
\node at (0.9,-0.5) {$p_{n-1}$};
\node at (1.1,0) {$\vdots$};

\node at (5.1,0) {$p_n$};

\filldraw[white] (2,0) circle (0.45cm) node {$w_1$};
\filldraw[white] (4,0) circle (0.45cm) node {$w_2$};
\small \draw (2,0) circle (0.45cm) node {$g-1$};
\draw (4,0) circle (0.45cm) node {$0$};
\end{tikzpicture}
.
\]
Form the commutative diagram
\begin{equation} \label{eqn:Cfibrediagram2}
\begin{tikzcd}
 &\displaystyle\coprod_{(\Gamma,\Gamma')} \overline{\mathcal{H}}_{(\Gamma, \Gamma')} \arrow[d] \arrow[r,"\coprod \xi_{(\Gamma,\Gamma')}"]
 & \overline{\mathcal{H}}_{g,d,r_1,...,r_k}\arrow[d]\\
 &\displaystyle\coprod_{\widehat {\Gamma_0}} \oM_{\widehat \Gamma_0} \times \omM_{0,k} \arrow[d] \arrow[r,"\coprod \xi_{\widehat{ \Gamma}_0} \times 1"] & \oM_{g,n + b}\times \omM_{0,k} \arrow[d]\\
 & \oM_{\Gamma_0}\times \omM_{0,k} \arrow[r,"\xi_{\Gamma_0} \times 1"]& \oM_{g,n} \times \omM_{0,k}
\end{tikzcd}
\end{equation}
and denote by $\tau_{(\Gamma,\Gamma')}$ the composite morphism 
$$
\overline{\mathcal{H}}_{(\Gamma, \Gamma')} \rightarrow \oM_{\Gamma_0}\times \omM_{0,k}
$$
as in the previous section.

\begin{lemma}
The pairs $(\Gamma,\Gamma')$ appearing in diagram \eqref{eqn:Cfibrediagram2} for which $\oH_{(\Gamma,\Gamma')}$ dominates $\oM_{\Gamma_0}\times \omM_{0,k}$ are the following: 
\begin{itemize}
    \item \textbf{type I}:
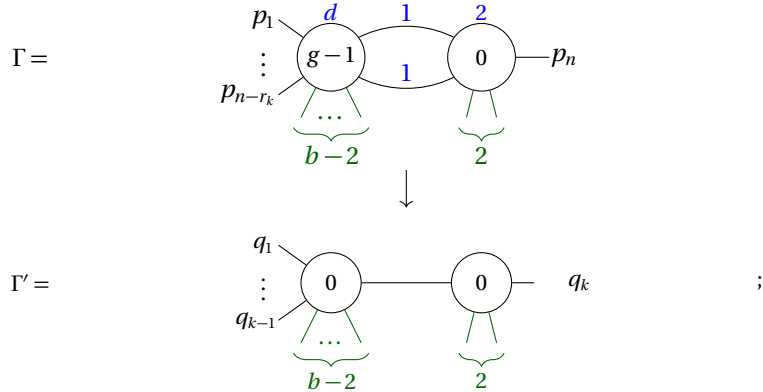
\[
\begin{tikzpicture}[baseline = -0.1 cm]
%\draw [help lines] (-5,-5) grid (5, 5);

%Here Gamma starts
\node at (-2,0) {$\Gamma=$};
\draw (2,0) to [out=45,in=135] (4,0);
\draw (2,0) to [out=-45,in=-135] (4,0);
\node[blue] at (3,0.6) {$1$};
\node[blue] at (3,-0.2) {$1$};

%black legs on g-1

\draw (2,0) -- (1.3,0.5);
\draw (2,0) -- (1.3,-0.5);

\node at (1.1,0.5) {$p_1$};
\node at (0.9,-0.5) {$p_{n-r_k}$};
\node at (1.1,0) {$\vdots$};

%green legs on g-1

\draw[black!60!green] (2,0) -- (1.6,-0.8);
\node[black!60!green] at (2,-0.8) {$\dots$};
\draw[black!60!green] (2,0) -- (2.4,-0.8);

\draw [black!60!green, decorate,decoration={brace,amplitude=5pt,mirror,raise=1pt}] (1.5,-0.9) -- (2.5,-0.9) ;
\node[black!60!green] at (2,-1.3) {$b-2$};

%black legs on 0 in Gamma

\draw (4,0) -- (4.9,0);

\node at (5.1,0) {$p_n$};

%green legs on 0 in Gamma

\draw[black!60!green] (4,0) -- (3.8,-0.8);
\draw[black!60!green] (4,0) -- (4.2,-0.8);

\draw [black!60!green, decorate,decoration={brace,amplitude=5pt,mirror,raise=1pt}] (3.7,-0.9) -- (4.3,-0.9) ;
\node[black!60!green] at (4,-1.3) {$2$};

%vertices in Gamma

\filldraw[white] (2,0) circle (0.45cm) node {$w_1$};
\filldraw[white] (4,0) circle (0.45cm) node {$w_2$};

\small \draw (2,0) circle (0.45cm) node {$g-1$};
\draw (4,0) circle (0.45cm) node {$0$};

\node[blue] at (2,0.6) {$d$};
\node[blue] at (4,0.6) {$2$};

%arrow down

\draw[->] (3,-1.5) -- (3,-2);

%Here Gamma' starts

\node at (-2,-3) {$\Gamma'=$};
\draw (2,-3) to (4,-3);

%black legs on the left vertex of Gamma'
\draw (2,-3) -- (1.3,-2.5);
\draw (2,-3) -- (1.3,-3.5);

\node at (1.1,-2.5) {$q_1$};
\node at (1,-3.5) {$q_{k-1}$};
\node at (1.1,-3) {$\vdots$};

%green legs on the left vertex of Gamma'

\draw[black!60!green] (2,-3) -- (1.6,-3.8);
\node[black!60!green] at (2,-3.8) {$\dots$};
\draw[black!60!green] (2,-3) -- (2.4,-3.8);

\draw [black!60!green, decorate,decoration={brace,amplitude=5pt,mirror,raise=1pt}] (1.5,-3.9) -- (2.5,-3.9) ;
\node[black!60!green] at (2,-4.3) {$b-2$};

% black leg on the right vertex of Gamma'

\draw (4,-3) -- (4.7,-3);

\node at (5.3,-3) {$q_{k}$};

% green legs on the right vertex of Gamma'

\draw[black!60!green] (4,-3) -- (3.8,-3.8);
\draw[black!60!green] (4,-3) -- (4.2,-3.8);

\draw [black!60!green, decorate,decoration={brace,amplitude=5pt,mirror,raise=1pt}] (3.7,-3.9) -- (4.3,-3.9) ;
\node[black!60!green] at (4,-4.3) {$2$};

%vertices of Gamma'

\filldraw[white] (2,-3) circle (0.4cm) node {$w_1$};
\filldraw[white] (4,-3) circle (0.4cm) node {$w_2$};
\draw (2,-3) circle (0.4cm) node {$0$};
\draw (4,-3) circle (0.4cm) node {$0$};

\node at (7.7,-3) {$\mathrm{;}$};
\end{tikzpicture}
\]
\item \textbf{type II}:
\[
\begin{tikzpicture}[baseline = -0.1 cm]
%\draw [help lines] (-4,-4) grid (7, 7);

%Here Gamma starts
\node at (-2,0) {$\Gamma=$};

%edges
\draw (2,0) to (4,1);
\draw (4,1) to (6,0);
\draw (2,0) to (4,-1);
\draw (4,-1) to (6,0);

%degrees of the edges
\node[blue] at (3,0.7) {$1$};
\node[blue] at (3,-0.2) {$1$};
\node[blue] at (5,0.7) {$1$};
\node[blue] at (5,-0.2) {$1$};

%black legs on g-1

\draw (2,0) -- (1.3,0.5);
\draw (2,0) -- (1.3,-0.5);

\node at (1.1,0.5) {$p_1$};
\node at (0.9,-0.5) {$p_{n-r_k}$};
\node at (1.1,0) {$\vdots$};

%green legs on g-1

\draw[black!60!green] (2,0) -- (1.6,-0.8);
\node[black!60!green] at (2,-0.8) {$\dots$};
\draw[black!60!green] (2,0) -- (2.4,-0.8);

\draw [black!60!green, decorate,decoration={brace,amplitude=5pt,mirror,raise=1pt}] (1.5,-0.9) -- (2.5,-0.9) ;
\node[black!60!green] at (2,-1.3) {$b-2$};

%black legs on the central 0 in Gamma

\draw (4,-1) -- (4.6,-1.5);

\node at (4.8,-1.6) {$p_n$};

%green legs on the right 0 in Gamma

\draw[black!60!green] (6,0) -- (6.9,0.2);
\draw[black!60!green] (6,0) -- (6.9,-0.2);

\draw [black!60!green, decorate,decoration={brace,amplitude=5pt,raise=1pt}] (7,0.3) -- (7,-0.3) ;
\node[black!60!green] at (7.4,0) {$2$};

%vertices in Gamma

\filldraw[white] (2,0) circle (0.45cm) node {$w_1$};
\filldraw[white] (4,1) circle (0.45cm) node {$w_2$};
\filldraw[white] (4,-1) circle (0.45cm) node {$w_3$};
\filldraw[white] (6,0) circle (0.45cm) node {$w_4$};

\small \draw (2,0) circle (0.45cm) node {$g-1$};
\draw (4,1) circle (0.45cm) node {$0$};
\draw (4,-1) circle (0.45cm) node {$0$};
\draw (6,0) circle (0.45cm) node {$0$};

\node[blue] at (2,0.6) {$d$};
\node[blue] at (4,1.6) {$1$};
\node[blue] at (4,-0.4) {$1$};
\node[blue] at (6,0.6) {$2$};

%arrow down

\draw[->] (4,-1.8) -- (4,-2.3);

%Here Gamma' starts

\node at (-2,-3) {$\Gamma'=$};

%edges 
\draw (2,-3) to (4,-3);
\draw (4,-3) to (6,-3);

%black legs on the left vertex of Gamma'
\draw (2,-3) -- (1.3,-2.5);
\draw (2,-3) -- (1.3,-3.5);

\node at (1.1,-2.5) {$q_1$};
\node at (0.5,-3.5) {$q_{k-1}$};
\node at (1.1,-3) {$\vdots$};

%green legs on the left vertex of Gamma'

\draw[black!60!green] (2,-3) -- (1.6,-3.8);
\node[black!60!green] at (2,-3.8) {$\dots$};
\draw[black!60!green] (2,-3) -- (2.4,-3.8);

\draw [black!60!green, decorate,decoration={brace,amplitude=5pt,mirror,raise=1pt}] (1.5,-3.9) -- (2.5,-3.9) ;
\node[black!60!green] at (2,-4.3) {$b-2$};

% black leg on the central vertex of Gamma'

\draw (4,-3) -- (4.6,-3.5);

\node at (4.8,-3.7) {$q_{k}$};

% green legs on the right vertex of Gamma'

\draw[black!60!green] (6,-3) -- (6.9,-2.8);
\draw[black!60!green] (6,-3) -- (6.9,-3.2);

\draw [black!60!green, decorate,decoration={brace,amplitude=5pt,raise=1pt}] (7,-2.7) -- (7,-3.3) ;
\node[black!60!green] at (7.4,-3) {$2$};

%vertices of Gamma'

\filldraw[white] (2,-3) circle (0.4cm) node {$w_1$};
\filldraw[white] (4,-3) circle (0.4cm) node {$w_2$};
\filldraw[white] (6,-3) circle (0.4cm) node {$w_3$};

\draw (2,-3) circle (0.4cm) node {$0$};
\draw (4,-3) circle (0.4cm) node {$0$};
\draw (6,-3) circle (0.4cm) node {$0$};

\node at (7.7,-3) {$\mathrm{;}$};

\end{tikzpicture}
\]
\item \textbf{type III}:
\[
\begin{tikzpicture}[baseline = -0.1 cm]
%\draw [help lines] (-4,-4) grid (7, 7);

%Here Gamma starts
\node at (-2,0) {$\Gamma=$};

%edges
\draw (2,0) to (4,1);
\draw (4,1) to (6,0);
\draw (2,0) to (4,-1);
\draw (4,-1) to (6,0);

%degrees of the edges
\node[blue] at (3,0.7) {$1$};
\node[blue] at (3,-0.2) {$1$};
\node[blue] at (5,0.7) {$1$};
\node[blue] at (5,-0.2) {$1$};

%black legs on the g-1

\draw (4,1) -- (3.05,1.2);
\draw (4,1) -- (3.4,1.8);

\node at (2.9,1.2) {$p_1$};
\node at (3.3,2) {$p_{n-1}$};
\draw[thick,dotted] (3,1.5) -- (3.1,1.7);

%green legs on g-1

\draw[black!60!green] (4,1) -- (5.05,1.2);
\draw[black!60!green] (4,1) -- (4.6,1.8);
\draw[thick,dotted,black!60!green] (4.85,1.45) -- (4.7,1.65);
\draw[black!60!green,decorate,decoration={brace,amplitude=10pt,raise=1pt},yshift=0pt] (4.7,1.9) -- (5.15,1.3);

\node[black!60!green] at (5.75,1.9) {$b-4$};

%black legs on the central 0 in Gamma

\draw (4,-1) -- (4.6,-1.5);

\node at (4.8,-1.6) {$p_n$};

%green legs on the right 0 in Gamma

\draw[black!60!green] (6,0) -- (6.9,0.2);
\draw[black!60!green] (6,0) -- (6.9,-0.2);

\draw [black!60!green, decorate,decoration={brace,amplitude=5pt,raise=1pt}] (7,0.3) -- (7,-0.3) ;
\node[black!60!green] at (7.4,0) {$2$};

%green legs on the left 0 in Gamma

\draw[black!60!green] (2,0) -- (1.1,0.2);
\draw[black!60!green] (2,0) -- (1.1,-0.2);

\draw [black!60!green, decorate,decoration={brace,amplitude=5pt,raise=1pt,mirror}] (1,0.3) -- (1,-0.3) ;
\node[black!60!green] at (0.6,0) {$2$};

%vertices in Gamma

\filldraw[white] (2,0) circle (0.45cm) node {$w_1$};
\filldraw[white] (4,1) circle (0.45cm) node {$w_2$};
\filldraw[white] (4,-1) circle (0.45cm) node {$w_3$};
\filldraw[white] (6,0) circle (0.45cm) node {$w_4$};

\draw (2,0) circle (0.45cm) node {$0$};
\small \draw (4,1) circle (0.45cm) node {$g-1$};
\draw (4,-1) circle (0.45cm) node {$0$};
\draw (6,0) circle (0.45cm) node {$0$};

\node[blue] at (2,0.6) {$2$};
\footnotesize \node[blue] at (4,1.6) {$d-1$};
\node[blue] at (4,-0.4) {$1$};
\node[blue] at (6,0.6) {$2$};

%arrow down

\draw[->] (4,-1.8) -- (4,-2.3);

%Here Gamma' starts

\node at (-2,-3) {$\Gamma'=$};

%edges 
\draw (2,-3) to (4,-3);
\draw (4,-3) to (6,-3);

%green legs on the left vertex of Gamma'

\draw[black!60!green] (2,-3) -- (1.1,-2.8);
\draw[black!60!green] (2,-3) -- (1.1,-3.2);

\draw [black!60!green, decorate,decoration={brace,amplitude=5pt,raise=1pt,mirror}] (1,-2.7) -- (1,-3.3) ;
\node[black!60!green] at (0.6,-3) {$2$};

% black leg on the central vertex of Gamma'

\draw (4,-3) -- (3.6,-3.8);
\node at (4,-3.8) {$\dots$};
\draw (4,-3) -- (4.4,-3.8);

\node at (3.5,-4) {$q_1$};
\node at (4.5,-4) {$q_{k}$};

%green legs on central vertex

\draw[black!60!green] (4,-3) -- (5.05,-2.8);
\draw[black!60!green] (4,-3) -- (4.6,-2.2);
\draw[thick,dotted,black!60!green] (4.85,-2.55) -- (4.7,-2.35);
\draw[black!60!green,decorate,decoration={brace,amplitude=10pt,raise=1pt},yshift=0pt] (4.7,-2.1) -- (5.15,-2.7);

\node[black!60!green] at (5.75,-2.1) {$b-4$};

% green legs on the right vertex of Gamma'

\draw[black!60!green] (6,-3) -- (6.9,-2.8);
\draw[black!60!green] (6,-3) -- (6.9,-3.2);

\draw [black!60!green, decorate,decoration={brace,amplitude=5pt,raise=1pt}] (7,-2.7) -- (7,-3.3) ;
\node[black!60!green] at (7.4,-3) {$2$};

%vertices of Gamma'

\filldraw[white] (2,-3) circle (0.4cm) node {$w_1$};
\filldraw[white] (4,-3) circle (0.4cm) node {$w_2$};
\filldraw[white] (6,-3) circle (0.4cm) node {$w_3$};

\draw (2,-3) circle (0.4cm) node {$0$};
\draw (4,-3) circle (0.4cm) node {$0$};
\draw (6,-3) circle (0.4cm) node {$0$};

\node at (7.7,-3) {$\mathrm{.}$};

\end{tikzpicture}
\]
\end{itemize} 
Conventions here are the same as in Lemma \ref{contributions in the reduction}.

Moreover, the upper square in diagram \eqref{eqn:Cfibrediagram2} is Cartesian when restricted to these components $\oH_{(\Gamma,\Gamma')}$.
\end{lemma}
\begin{proof}
The proof follows exactly the lines of the proof of \cite[Proposition 7]{CPS}.
\end{proof}

To compute the degree of $\tau_{g,d ,n,r_1,...,r_k}:\oH_{g,d,r_1,...,r_k}\to \oM_{g,n}\times\omM_{0,k}$, the excess intersection formula gives

\begin{equation*}
    \deg(\tau_{g,d ,n,r_1,...,r_k})=\sum_{(\Gamma,\Gamma')} \deg\left(\left\{ c(N_{(\Gamma,\Gamma')}) \cap s(\overline{\mathcal{H}}_{(\Gamma, \Gamma')},\overline{\mathcal{H}}_{g,d,r_1,...,r_k}) \right\}_{\mathrm{dim}-2}\right)
\end{equation*}
% $$
%     \big( \oM_{\Gamma_0}\times \oM_{0,n-r_{\mathrm{tot}}+k} \big) . \overline{\mathcal{H}}_{g,d ,n,r_1,...,r_k} = 
%     \sum_{(\Gamma,\Gamma')} \left\{ c(N_{(\Gamma,\Gamma')} \cap s(\overline{\mathcal{H}}_{(\Gamma, \Gamma')},\overline{\mathcal{H}}_{g,d,n,r_1,...,r_k}) \right\}_{\mathrm{dim}-2}
% $$
where the sum is over all the pairs $(\Gamma,\Gamma')$ appearing in diagram \eqref{eqn:Cfibrediagram2}, the line bundle $N_{(\Gamma,\Gamma')}=\tau_{(\Gamma,\Gamma')}^* N_{\xi_{\Gamma_0} \times 1}$ is the pullback under $\tau_{(\Gamma,\Gamma')}$ of the normal bundle of $\xi_{\Gamma_0} \times 1$ and the terms in the sum are all in homological dimension $\mathrm{dim}-2$. On the right hand side, $\deg$ denotes the degree of the given class given upon pushforward to $\oM_{\Gamma_0}\times \omM_{0,k}$, as a multiple of the fundamental class.

Following the computations done in \cite[Proposition 7]{CPS}, we see that:

\begin{itemize}
\item from the $ (\Gamma, \Gamma')$ of type I,  we get the contribution 
$$
- b[g,l,r_1,...,r_k]! \ \mathsf{Tev}_{g-1,l+1,r_1,...r_{k-1},r_k+1},
$$
\item from the $ (\Gamma, \Gamma')$ of type II, we get the contribution
$$
2 b[g,l,r_1,...,r_k]! \ \mathsf{Tev}_{g-1,l+1,r_1,...r_{k-1},r_k+1},
$$
\item and from the $ (\Gamma, \Gamma')$ of type III, we get the contribution
$$
b[g,l,r_1,...,r_k]! \ \mathsf{Tev}_{g-1,l,r_1,...r_{k-1},r_k-1}.
$$
\end{itemize}
Proposition \ref{prop recursion} follows by summing these three terms. 

\section{Genus 0} \label{The genus 0 case}
Now, we consider the genus $0$ case. The proof here differs slightly from that of the analogous \cite[Proof of Proposition 8]{CPS}. 

Let $g=0$, $\ell \geq 0$ and $r_1,...,r_k \in \mathbb{Z}_{\geq 1}$ positive integers such that the Equations \eqref{dimensional constraint}, \eqref{conditions on d}  and \eqref{condition on k} hold with 
$$
\mu_h=(\underbrace{1,...,1}_{r_h})
$$
for $h=1,...,k$. 

Let us first explain the argument. A map $\pi:\bP^1\to\bP^1$ of degree $d$ is given by a pair of polynomials $s_0,s_1$ of degree $d$ up to simultaneous scaling. We may parametrize the coefficients of the $f_i$ by the projective space $\mathbb{P}(H^0(\mathbb{P}^1,\mathcal{O}(d)) \oplus H^0(\mathbb{P}^1,\mathcal{O}(d)))\cong\bP^{2d+1}$. Then, the $n=2d+1$ incidence conditions on $f$ are cut out by hyperplanes $\bP^{2d+1}$, which are expected to intersect in a single point giving a unique map $\pi$ of degree $d$. The following shows that this is indeed the case when the $p_i,q_j$ are general.

Formally, consider the closed subscheme 
$$
V_{d,r_1,...,r_k} \subset \mathbb{P}(H^0(\mathbb{P}^1,\mathcal{O}(d)) \oplus H^0(\mathbb{P}^1,\mathcal{O}(d)) ) \times (\mathbb{P}^1)^n \times (\mathbb{P}^1)^{k} 
$$
defined by the condition that the $\mathbb{C}$-point
$$
([s_0:s_1],p_1,...,p_n,q_1=[q_1^0:q_1^1],...,q_{k}=[q_{k}^0:q_{k}^1])
$$
of $ \mathbb{P}(H^0(\mathbb{P}^1,\mathcal{O}(d)) \times H^0(\mathbb{P}^1,\mathcal{O}(d)) ) \times (\mathbb{P}^1)^n \times (\mathbb{P}^1)^{k}
$ belongs to $V_{d,r_1,...,r_k}$ if and only if $q_{j}^1s_0(p_i)-q_{j}^0s_1(p_i)=0$ for $1 \leq j \leq k$ and $n- \sum_{h=j}^k r_h +1 \leq i \leq n- \sum_{h=j+1}^{k} r_h$.

%if we identify an element $[s_0;s_1]$ of $\mathbb{P}(H^0(\mathbb{P}^1,\mathcal{O}(d))$ with the corresponding map $\pi: \mathbb{P}^1 \rightarrow \mathbb{P}^1$ of degree $\leq d$, then for the $\mathbb{C}$-point
%$$
%(\pi,p_1,...,p_n,q_1,...,q_{n-r_{\mathrm{tot}}+k})
%$$
%of $ \mathbb{P}(H^0(\mathbb{P}^1,\mathcal{O}(d)) \times H^0(\mathbb{P}^1,\mathcal{O}(d)) ) \times (\mathbb{P}^1)^n \times (\mathbb{P}^1)^{n-\sum_{h=1}^k+k}
%$ we have 
%$$
%\begin{bmatrix}
%([s_0,s_1],p_1,...,p_n,q_1,...,q_{n-\sum_{h=%1}^k r_h+k})\\
%\mathrm{belongs} \  \mathrm{to} \ %V_{d,r_1,...,r_k}
%\end{bmatrix}
%\iff 
%\begin{bmatrix}
%\pi (p_i)= q_i \ \mathrm{for} \\
%1 \leq i \leq n- r_{\mathrm{tot}} \ %\mathrm{and}  \\
%\pi( p_i )= q_{n-k+j} \ \mathrm{for} \\
%1 \leq j \leq k \ \mathrm{and} \\ 
%n- \sum_{h=j}^k r_h +1 \leq i \leq n- %\sum_{h=j+1}^{k} r_h 
%\end{bmatrix}
%.
%$$
Also, let 
$$
V_{d,r_1,...,r_k}^\circ \subset V_{d,r_1,...,r_k}
$$
be the open subscheme consisting of points 
$$
(\pi=[s_0:s_1],p_1,...,p_n,q_1,...,q_{k})
$$
where the morphism $\pi$ has exactly degree $d$, is simply branched (i.e., the ramification index of $\pi$ is $\leq 2$ at every point of $\mathbb{P}^1$ and different ramification points in the domain are mapped to different branch points in the target) and the branch points of $\pi$ are distinct from the $q_1,...,q_{k}$. 

Clearly,
$$
V_{d,r_1,...,r_k}^\circ \neq \emptyset.
$$

We have the following commutative triangle:
\begin{equation} \label{def of p}
\begin{tikzcd} 
 V_{d,r_1,...,r_k} \arrow[swap]{dr}{p} \arrow[r,hook] &\mathbb{P}(H^0(\mathbb{P}^1,\mathcal{O}(d)) \oplus H^0(\mathbb{P}^1,\mathcal{O}(d)) ) \times (\mathbb{P}^1)^n \times (\mathbb{P}^1)^{k} \arrow{d}{\mathrm{projection}}\\
& (\mathbb{P}^1)^n \times (\mathbb{P}^1)^{k} 
\end{tikzcd}
\end{equation}

The next lemma implies that $\tau_{0,\ell,r_1,...,r_k}$ has degree $1$.

\begin{lemma}\label{genus0_inj}
The restriction $p|_{V_{d,r_1,...,r_k}^\circ}$ of the map in Diagram \ref{def of p} to $V_{d,r_1,...,r_k}^\circ$ is injective on closed points. Moreover, the set-theoretic image $p({V_{d,r_1,...,r_k}^\circ})$ is dense in $(\mathbb{P}^1)^{n} \times (\mathbb{P}^1)^{k}$.
\end{lemma}
\begin{proof}
%First of all we prove that it is bijective on %$\mathbb{C}$-points (and thus bijective on all %points). Riemann Existence Theorem tells us %that it is surjective. 
First, we prove injectivity. Let
$$
\pi_1,\pi_2: \mathbb{P}^1 \rightarrow \mathbb{P}^1
$$
be two degree $d$ maps in the fiber under $p$ of $(p_1,...,p_n,q_1,...,q_{k})$. Then, the quotient map
$$
f=\frac{\pi_1}{\pi_2}: \mathbb{P}^1 \rightarrow \mathbb{P}^1
$$
has degree at most $2d=n-1$ and the preimage of $1$ under $f$ contains at least the $n$ points $p_1,...,p_n$. Therefore, $f$ is the constant map equal to $1$ and $\pi_1=\pi_2$. This shows that $p|_{V_{d,r_1,...,r_k}^0}$ is injective on closed points. 

The second claim follows from the fact that 
$$
\mathrm{dim}(V_{d,r_1,...,r_k}^\circ)\geq \mathrm{dim}\big( (\mathbb{P}^1)^{n} \times (\mathbb{P}^1)^{k} \big)
$$
combined with the first claim.
%Since the restriction of $\mathrm{p}|_{V_{d,r_1,...,r_k}^0}$ to each $\mathbb{C}$-fiber is an isomorphism, by Lemma \cite[\href{https://stacks.math.columbia.edu/tag/05VH}{Tag 05VH}]{stacks-project}, we also know that $\mathrm{p}|_{V_{d,r_1,...,r_k}^0}$ is an unramified monomorphism. Finally, apply Lemma \cite[\href{https://stacks.math.columbia.edu/tag/04XV}{Tag 04XV}]{stacks-project} and the fact that $(\mathbb{P}^1)^n \times (\mathbb{P}^1)^{n-\sum_{h=1}^k r_h+k}$ is reduced
%to see that $\mathrm{p}|_{V_{d,r_1,...,r_k}^0}$ must actually be an isomorphism.
\end{proof}

\begin{proof}[Proof of Proposition \ref{prop for g=0}]
Lemma \ref{genus0_inj} shows that a general choice of points on the source and target curve gives rise a unique cover $\pi:\bP^1\to\bP^1$ counted in $\Tev_{0,\ell,r_1,\ldots,r_k}$, so the conclusion is immediate.
\end{proof}

\section{Explicit solution to the recursion} \label{Explicit formulas} 
\subsection{Results for \texorpdfstring{$\ell\geq 0$}{l>=0}}
We start with case $\ell \geq 0$, in which our formulas are simpler.
\begin{proposition}
Suppose $g \geq 0$, $\ell \geq 0$ and $r_1,...,r_k \in \mathbb{Z}_{\geq 1}$. Assume that the Equations \eqref{dimensional constraint}, \eqref{conditions on d}  and \eqref{condition on k} hold with 
$$
\mu_h=(\underbrace{1,...,1}_{r_h})
$$
for $h=1,...,k$. Then, we have
\begin{equation*}
\mathsf{Tev}_{g,l,r_1,...,r_k}=
\begin{cases}
2^g- \displaystyle\sum_{h=1}^k \displaystyle\sum_{j=0}^{r_h-2} {g \choose j} &\text{for } \ell =0,
\\  \mathsf{Tev}_{g,0,r_1-\ell,...,r_k-\ell} & \text{for } \ell < \max(r_1,...,r_k),
\\ 2^g  & \text{for } \ell\geq \max(r_1,...,r_k).
\end{cases}
\end{equation*}
Here, for $\ell < \max(r_1,...,r_k)$, we set:
$$
\mathsf{Tev}_{g,0,r_1-\ell,...,r_k-\ell}:=\mathsf{Tev}_{g,0,r_{i_1}-\ell,...,r_{i_t}-\ell,1,...,1}
$$
where the indices $h \in \{i_1,...,i_t\}$ are those for which $r_h > \ell$ and at the end there are $g+3-\sum_{h=1}^t (r_{i_h} -\ell)$ numbers $1$.
\end{proposition}
\begin{proof}
We start by noticing that for $\ell=\max(r_1,...,r_k)$, the two definitions of the right hand side coincide. The proof of the proposition is by induction on $g$. 

{\tiny{$\blacksquare$}} Suppose that \textbf{$g=0$}. Then, by Proposition \ref{prop for g=0}, the left hand side is equal to $1$, while for the left hand side we have to distinguish three cases: 
\begin{itemize} 
 \item If $\ell=0$, then $(d,n)=(1,3)$, and thus $k =  3$ and $r_h=1$ for all $h \leq k$. The right hand side is then equal to $1$.
 \item If $\ell < \max(r_1,...,r_k)$, then there are at most 3 and at least 1 indices $h \in \{1,...,k\}$ for which $r_h>\ell$. Moreover, for these $h$, it must be $r_h=\ell +1$. The right hand side is therefore equal to $\mathsf{Tev}_{0,0,1,1,1}$, which is equal to $1$ by Proposition \ref{The genus 0 case}.
 \item If $\ell \geq \max(r_1,...,r_k)$, then we have $2^g=1$ on the right hand side.
\end{itemize}

{\tiny{$\blacksquare$}} Suppose that $g \geq 1$. Then, we distinguish two cases:
\begin{itemize} 
\item Suppose $\ell \geq \max(r_1,...,r_k)$. Then, this inequality persists during the recursion, and we have
\begin{align*}
\mathsf{Tev}_{g,\ell,r_1,...,r_k}=& \mathsf{Tev}_{g-1,\ell,r_1,...,r_{k-1},r_k-1}+\mathsf{Tev}_{g-1,\ell+1,r_1,...,r_{k-1},r_k+1}
=2^{g-1}+2^{g-1}=2^g
\end{align*}
where in the second equality we have used the inductive hypothesis. We leave to the reader to check that the Equations \eqref{dimensional constraint}, \eqref{conditions on d}  and \eqref{condition on k} also persist during the recursion.
\item Suppose instead that $ \ell < \max(r_1,...,r_k).$ Then the non-strict inequality persists during the recursion and, for $r_k<d$, we can write 
\begin{align*}
\mathsf{Tev}_{g,\ell,r_1,...,r_k} 
& =\mathsf{Tev}_{g-1,\ell,r_1,...,r_{k-1},r_k-1} + \mathsf{Tev}_{g-1,\ell+1,r_1,...,r_{k-1},r_k+1}\\
&=\mathsf{Tev}_{g-1,0,r_1-\ell,...,r_{k-1}-\ell,r_k-1-\ell} + \mathsf{Tev}_{g-1,0,r_1-\ell-1,...,,r_{k-1}-\ell-1,r_k-\ell}\\
& = 2^{g-1}- \sum_{h=1}^{ k-1} \sum_{j=0}^{r_h-\ell-2}{ g-1 \choose j}- \sum_{j=0}^{r_k-\ell-3} {g-1 \choose j} \\
&+ 2^{g-1} - \sum_{h=0}^{k-1 } \sum_{j=0}^{r_h-\ell-3}{g-1 \choose j}- \sum_{j=0}^{r_k-\ell-2} { g-1 \choose j} \\
& = 2^g - \sum_{h=0}^{ k-1 } \sum_{j=0}^{r_h-\ell-2} \Bigg[ { g-1 \choose j-1} + {g-1 \choose j} \Bigg]- \sum_{j=0}^{r_k-\ell-2} \Bigg[ { g-1 \choose j-1} + {g-1 \choose j} \Bigg] \\
& = 2^g - \sum_{h=0}^{ k } \sum_{j=0}^{r_h-\ell-2} { g \choose j}.
\end{align*}
The case $r_k=d$ requires special attention, because in this case $d[g-1,\ell+1]< r_k+1$ and 
$$
\mathsf{Tev}_{g-1,\ell+1,r_1,...,r_{k-1},r_k+1}=0.
$$
A priori, we cannot apply the inductive hypothesis in the second and third equalities above. However, the proof goes through also in this case because when $r_k=d$, by condition \eqref{dimensional constraint}, it must be that $r_h\leq 1+\ell$ for all $h<k$ and thus also
$$
 2^{g-1} - \sum_{h=0}^{k-1 } \sum_{j=0}^{r_h-\ell-3}{g-1 \choose j}- \sum_{j=0}^{r_k-\ell-2} { g-1 \choose j}=2^{g-1}-\sum_{j=0}^{g-1}{ g-1 \choose j}=0.
$$
\end{itemize}
\end{proof}

\subsection{Results for general $\ell$}
Finally, the master result which also covers the $\ell \geq 0$ case takes the following form:

\begin{theorem}\label{master for case muij=1}
Suppose $g \geq 0$, $\ell \in \mathbb{Z}$ and $r_1,...,r_k \in \mathbb{Z}_{\geq 1}$. Assume that the inequalities \eqref{dimensional constraint}, \eqref{conditions on d} and \eqref{condition on k} hold with 
$$
\mu_h=(\underbrace{1,...,1}_{r_h})
$$
for $h=1,...,k$. Then,
\begin{align*}
\mathsf{Tev}_{g,l,r_1,...,r_k}= & \ 2^g -2\sum_{i=0}^{-\ell-2} {g \choose i} + \left(- \ell-k-2+n\right){g \choose {-\ell-1}} \\
&+\left(\ell -k + \sum_{h=1}^k \delta_{r_h,1}\right) {g \choose {-\ell}}-\sum_{h=1}^k \sum_{i=- \ell+1}^{r_h- \ell -2} {g \choose i}.
\end{align*}
Here $\delta_{r_h,1}$ is the Kronecker delta.
\end{theorem}
\begin{proof}
We prove also this statement by induction on $g$.

{\tiny{$\blacksquare$}} If $g=0$, then the left hand side is equal to 1, and all terms of the right hand side except $2^g=1$ vanish when $\ell>0$. The only other possibility is $\ell=0$, in which case $d=\ell+1=1$ and $r_h=1$ for $h=1,2,\ldots,k$, so the only remaining term $$\left(\ell -k + \sum_{h=1}^k \delta_{r_h,1}\right)$$ also vanishes.

% then it must be that $\ell \geq 0$, and the left hand side is
% $$
% 1+ \left(-k+\sum_{h=0}^k \delta_{r_h,1}\right)\delta_{\ell,0}=1
% $$
% being necessarily $r_h=1$ for all $h=1,...,k$ when $\ell=0$.

{\tiny{$\blacksquare$}} Suppose instead that $g \geq 1$. We may assume that $r_1 \leq ... \leq r_k$. Also, since in the case $r_1=...=r_{k-1}=1$ the formula reduces to \cite[Theorem 8]{CPS}, we may assume that $k>1$ and $r_{k-1}>1$. We then write 
$$
\mathsf{Tev}_{g,\ell,r_1,...,r_k} 
=\mathsf{Tev}_{g-1,\ell,r_1,...,r_{k-1},r_k-1} + \mathsf{Tev}_{g-1,\ell+1,r_1,...,r_{k-1},r_k+1}
$$
and, for $r_k \neq d$, we can apply to inductive hypothesis to rewrite the right hand side as
\begin{equation}\label{inductive step main formula}
\begin{aligned} 
&2^{g-1}-2 \sum_{j=0}^{-\ell-2} {g-1 \choose j} +\left(-\ell-k-2+\sum_{h=1}^{k-1} r_h+r_k-1\right){g-1 \choose -\ell-1} \\
&+\left(\ell -k +\sum_{h=1}^{k-1} \delta_{r_h,1} + \delta_{r_k-1,1} \right){g-1 \choose -\ell}- \sum_{h=1}^{k-1} \sum_{j=-\ell+1}^{r_h-\ell-2} {g-1 \choose j}  -\sum_{j=-\ell+1}^{r_k-\ell-3} {g-1 \choose j} \\&+2^{g-1}-2 \sum_{j=0}^{-\ell-3} {g-1 \choose j} +\left(-\ell-1-k-2+\sum_{h=1}^{k-1} r_h+r_k+1\right){g-1 \choose -\ell-2} \\
&+\left(\ell+1 -k +\sum_{h=1}^{k-1} \delta_{r_h,1} + \delta_{r_k+1,1} \right){g-1 \choose -\ell-1} - \sum_{h=1}^{k-1} \sum_{j=-\ell}^{r_h-\ell-3} {g-1 \choose j}  -\sum_{j=-\ell}^{r_k-\ell-2} {g-1 \choose j}.
\end{aligned}
\end{equation}
Now, transferring a term ${g-1 \choose -\ell-1}$ from the last row to the second row and combining the corresponding terms in rows 1 and 3 and in 2 two and 4, we get the desired formula.

When $r_k=d$, we have $d[g-1,\ell+1]<r_k+1$, so the inductive hypothesis may not be applied directly. In this case, we have
$$
\mathsf{Tev}_{g-1,\ell+1,r_1,...,r_{k-1},r_k+1}=0.
$$
Moreover, by condition 
 \eqref{dimensional constraint}, it must be that$\ell >-2$ and $r_h<\ell+2$ for all $h<k$. Also, since we are assuming $r_{k-1}>1$, it must be that $\ell \geq 0$, and thus the sum of the terms in the third and fourth rows of Equation \ref{inductive step main formula} is
$$
2^{g-1}- \sum_{j=0}^{g-1} {g-1 \choose j}=0.
$$
This concludes the proof.
\end{proof}

\begin{proof}[Proof of Theorem \ref{master formula}]
Combine Theorem \ref{master for case muij=1} and Proposition \ref{prop reduction}.
\end{proof}

\section{Schubert calculus formula via limit linear series}\label{sec:schubert}

In this section, we explain a second approach to the generalized Tevelev degrees and establish the alternate formula of Theorem \ref{main_thm_schubert}.

The strategy is as follows. We wish to count degree $d$ morphisms $\pi:C\to\bP^1$ satisfying certain incidence conditions. Such a morphism is determined by a line bundle $\cL$ on $C$ of degree $d$, along with two linearly independent sections $f_0,f_1\in H^0(C,\cL)$ spanning a rank 1 linear series $V=\langle f_0,f_1\rangle \subset H^0(C,\cL)$. The condition that $f(p)=q$ for $p\in C$ and $q\in\bP^1$ translates to a linear condition on the sections $f_0,f_1$ upon evaluation at $p$.

We then count $V,f_0,f_1$ on $C$ satisfying the desired properties by considering the limit of this data on a degeneration of $C$. Our approach is via the theory of limit linear series \cite{EH}, with which we assume some familiarity.

\subsection{Degeneration setup}

For notational convenience, we will assume that henceforth $\mu_h=(1)^{|\mu_h|}=(1)^{r_h}$ for $h=1,2,\ldots,k$; by Proposition \ref{prop reduction}, we may reduce to this case, but in fact, the proof that follows works equally well in the general case.

We fix, for the entirety of the proof of Theorem \ref{main_thm_schubert}, a 1-parameter family $\psi:\cC\to B,p_1,...,p_n:B \rightarrow \cC$ of $n$-pointed genus $g$ curves. $B$ may be taken to be a complex analytic disk or the spectrum of a discrete valuation ring with residue field $\bC$. We assume that the total space $\cC$ is smooth, the generic fiber $\cC_{\eta}$ is a general curve over the residue field of $\eta\in B$ (or a 1-parameter family of general curves over a punctured disc), and the special fiber $C_0$ has the following form.

The nodal curve $C_0$ contains a rational component (\textit{spine}) $C_{\spine}$, to which general elliptic tails $E_1,\ldots,E_g$ are attached at general points $s_1,\ldots,s_g\in C_{\spine}$, respectively, and rational tails $R_1,\ldots,R_k$ are attached at general points $t_1,\ldots,t_k\in C_{\spine}$, respectively. 

We assume that the $r_h$ pre-images of the target point $q_{h}\in\bP^1$ are situated at general points of the rational tail $R_h$ for $h=1,\ldots,k$ (see Figure \ref{the curve C0}). The family may be regarded as arising from the collision (in the limit) of these $r_h$ pre-images at the point $t_h$ for each $h$. Note that we allow $r_h=1$, so the pointed curve $C_0$ may not be stable.

\begin{figure}
\begin{center}
\begin{tikzpicture}[xscale=0.4,yscale=0.30] [xscale=0.36,yscale=0.36]
	%\draw [help lines] (0,0) grid (30, 15);
			
    %the spine
    \node at (1,16) {$C_{\mathrm{sp}}$};
	\draw [ultra thick, black] (1,0) to (1,15);
	
	%the rational tails
	
	\draw [ultra thick, black] (-0.5,1) to (15,1);
	\node at (16,1) {$R_1$};
	\node at (1,1) {$\bullet$};
	\node at (0.3,1.8) {$t_1$};
	
	\draw [ultra thick, black] (-0.5,5) to (15,5);
	\node at (16,5) {$R_k$};
	\node at (1,5) {$\bullet$};
	\node at (0.3,5.8) {$t_k$};
	
	\node at (16,3) {$\vdots$};
	
	%the points on the rational tails
	
	%R_1
	\node at (3,1) {$\bullet$};
	\node at (4,0) {$p_{1}$};
	\node at (7.5,0) {$\dots$};
	\node at (10,1) {$\bullet$};
	\node at (11,0) {$p_{r_1}$};
	
	%R_k
	\node at (3,5) {$\bullet$};
	\node at (4,4) {$p_{n-r_k+1}$};
	\node at (7.5,4) {$\dots$};
	\node at (10,5) {$\bullet$};
	\node at (10,4) {$p_n$};

	%the genus 1 tails
	
	\draw [ultra thick, black] (-0.5,7) to [out=0,in=180] (8,7) to [out=0,in=180] (13,8) to [out=0,in=180] (15,8);
	
	\draw [ultra thick, black] (-0.5,10) to [out=0,in=180] (8,10) to [out=0,in=180] (13,11) to [out=0,in=180] (15,11);
	
	\node at (16,8) {$E_1$};
	\node at (16,10) {$\vdots$};
	\node at (16,11) {$E_g$};
	
	\node at (1,7) {$\bullet$};
	\node at (1,10) {$\bullet$};
	\node at (0.3,7.8) {$s_1$};
	\node at (0.3,10.8) {$s_g$};
	
	%points on Csp
	
	% \node at (1,12) {$\bullet$};
	% \node at (1.7,11.7) {$p_1$};
	
	% \node at (1.7,13) {$\vdots$};
	
	% \node at (1,14) {$\bullet$};
	% \node at (2.5,14) {$p_{n-r_{\mathrm{tot}}}$};

\end{tikzpicture}
\caption{The curve $C_0$.}\label{the curve C0}
\end{center}
\end{figure}
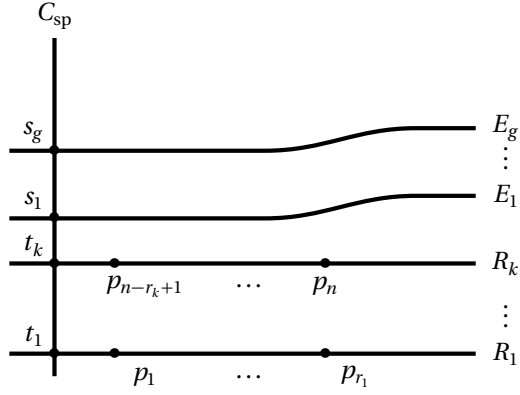

Let $G^{1}_d(\cC/B)\to B$ be the proper moduli space of (possibly crude) limit linear series of rank 1 and degree $d$ on the fibers of $\psi$. We recall some aspects of the construction, see e.g. \cite[Proposition 3.2.5]{lieblichosserman} for a modern treatment with further details.

Let $\Pic_{d,\spine}(\cC/B)$ be the relative Picard scheme of line bundles of degree $d>0$ on the fibers of $\psi$, where on the special fiber of $\cC$ we impose the finer requirement that the line bundle have degree $d$ on $C_{\spine}$ and 0 on all other components. Then, $\cC\times_B\Pic_{d,\spine}(\cC/B)$ is equipped with a Poincar\'{e} line bundle $\cL_{\spine}$, as well as twists $\cL_{i}$ for all other components $C_i\subset C_0$, which have the property that the restriction of $\cL_i$ to $C_0\times[\cL_i]$ has degree $d$ on $C_i$ and degree 0 on all other components.

From here, $G^{1}_d(\cC/B)$ may be constructed as a closed subscheme (with conditions defined by the vanishing of sections at the nodes of $C_0$) of a product of Grassmannian bundles $\Gr(2,\cW_i)$ over $\Pic_{d,\spine}(\cC/B)$. Here, $\cW_i$ is the vector bundle on $\Pic_{d,\spine}(\cC/B)$ whose fiber is given by the space $H^0(C,\cL)$ at a point over the generic fiber of $B$ and $H^0(C_i,\cL_i)$ over the special fiber. In particular, $G^{1}_d(\cC/B)$ carries tautological rank 2 bundles $\cV_i$ for each component $C_i\subset C_0$, which agree over the generic fiber of $B$, and are given by the aspect of the linear series on $C_i$ over the special fiber.

Now, let $\Coll^1_d(\cC/B)\to G^{1}_d(\cC/B)$ the $\bP^3$-bundle associated to $\cV_{\spine}^{\oplus 2}$. A point of $\Coll^1_d(\cC/B)$ parametrizes a point of $G^{1}_d(\cC/B)$ and a pair of sections $[f_0:f_1]$, taken up to simultaneous scaling, of the underlying linear series $V$ (on the generic fiber of $\psi$) or of the $C_{\spine}$-aspect $V_{\spine}$ (on the special fiber). Such a point may be regarded as a \textit{complete collineation} $\bC^2\to V$ or $\bC^2\to V_{\spine}$ in the sense of \cite{vainsencher}. We will refer to the divisor on $\Coll^1_d(\cC/B)$ parametrizing complete collineations where $f_0,f_1$ are linearly dependent as the \textbf{rank 1 locus}.

Fix, for the rest of the proof of Theorem \ref{main_thm_schubert}, general points $q_1,\ldots,q_{k}\in\bP^1$. 

\begin{definition}
We define a point $(V,[f_0:f_1])$ of the generic fiber of $\Coll^1_d(\cC/B)\to B$ to be a \textbf{Tevelev point} if $[f_0:f_1]$ defines a degree $d$ cover $\pi:C\to\bP^1$ satisfying the required incidence conditions of Definition \ref{def:hur_tev}(iv).
\end{definition}
%\Ale{Don't you also want that at the points $p_i$ your map is unramified ? that's why I proposed $\pi \in \cH_{0,d,n,r_1,...,r_k}$ before. }

By definition, the number of Tevelev points on the generic fiber is equal to $\Tev_{g,\ell,r_1,\ldots,r_k}$. By assumption, Tevelev points of $\Coll^1_d(\cC/B)$ define covers $f$ of degree $d$, so in particular are supported away from the rank 1 locus.

In order to enumerate Tevelev points, we compute their limits in the special fiber.

\begin{definition}
A point of $\Coll^1_d(\cC/B)_0$ is a \textbf{limit Tevelev point} if it lies in the closure of the locus of Tevelev points on the generic fiber of $\Coll^1_d(\cC/B)$
\end{definition}

The strategy of proof of Theorem \ref{main_thm_schubert} is as follows. In \S\ref{lim_tev_char}, we first give set-theoretic restrictions on limit Tevelev points, then show that all points of $\Coll^1_d(\cC/B)_0$ satisfying these conditions are in fact limit Tevelev. Moreover, the scheme-theoretic closure of the locus of Tevelev points is \'{e}tale over $B$, so it suffices to enumerative limit Tevelev points. This enumeration is then carried out in \S\ref{lim_tev_enum}.

\subsection{Characterization of limit Tevelev points}\label{lim_tev_char}

The following claims, Propositions \ref{prop:limit_tev_constraints}, \ref{prop:limit_tevelev_generic}, and \ref{prop:limit_tevelev_tails}, give constraints on a point $(V,[f_0:f_1])\in\Coll^1_d(\cC/B)_0$ to be limit Tevelev.

\begin{proposition}\label{prop:limit_tev_constraints}
Consider a limit Tevelev point $(V,[f_0:f_1])\in\Coll^1_d(\cC/B)_0$ and its aspect $V_{\spine}$ on $C_{\spine}$. Then, $(V_{\spine},[f_0:f_1])$ has the following properties.
\begin{enumerate}
    \item[(1)] $V_{\spine}$ is ramified at the node $s_j$ for $j=1,2,\ldots,g$, that is, at each $s_j$, $V_{\spine}$ has a non-zero section vanishing to order 2.
    %\item[(2)] $[f_0:f_1]$ is incident at $q_i\in\bP^1$ when evaluated at $p_i$ for $i=1,2,\ldots,n-r_{\tot}$, by which we mean that if $q_{i}=[q^0_i:q^1_i]$, then 
    %\begin{equation*}
    %    q^1_i\cdot f_0-q^0_i\cdot f_1\in H^0(\bP^1,\cO(d))
   % \end{equation*}
   % vanishes at $p_i$.
    \item[(2)] $[f_0:f_1]$ is incident at $q_{h}\in\bP^1$ to order $r_h$ when evaluated at $t_h$ for $h=1,2,\ldots,k$. That is,
    \begin{equation*}
        q^1_{h}\cdot f_0-q^0_{h}\cdot f_1\in H^0(\bP^1,\cO(d))
    \end{equation*}
    vanishes at $t_h$ to order $r_h$.
    \item[(2')] $V_{\spine}$ contains a \textit{non-zero} section vanishing at $t_h$ to order $r_h$ for $h=1,...,k$.
    \end{enumerate}
\end{proposition}

\begin{remark}
Note that if $f_0,f_1$ vanish at $t_h$ to order $r_h$, then (2) is satisfied for any $q_{h}\in\bP^1$, but (2') may not be.
\end{remark}

\begin{proof}
By Riemann-Hurwitz formula, the vanishing sequence of the aspect $V_{E_i}$ is at most $(d-2,d)$ at $s_i$, so because $V$ is a (possibly crude) limit linear series, $V_{\spine}$ must have vanishing sequence at least $(0,2)$, giving (1).

Regarding $t_h$ as the simultaneous limit of the $r_h$ points
\begin{equation*}
p_{n-r_{h}-r_{h+1}-\cdots-r_{k}+1},\ldots,p_{n-r_{h+1}-\cdots-r_{k}}
\end{equation*}
on the generic fiber, the condition that
\begin{equation*}
        q^1_{h}\cdot f_0-q^0_{h}\cdot f_1\in H^0(\bP^1,\cO(d))
\end{equation*}
vanishes on the divisor given by the sum of these $r_h$ points has flat limit in $\Coll^1(\cC/B)$ given by the condition that it vanishes on the divisor $r_h\cdot t_h$. This yields (2).

Condition (2') is immediate from (2), unless the section
\begin{equation*}
        q^1_{h}\cdot f_0-q^0_{h}\cdot f_1\in H^0(\bP^1,\cO(d))
\end{equation*}
is identically zero. However, the condition on the generic fiber that $V$ contains a 1-dimensional subspace of sections vanishing at
\begin{equation*}
p_{n-r_{h}-r_{h+1}-\cdots-r_{k}+1},\ldots,p_{n-r_{h+1}-\cdots-r_{k}}
\end{equation*}
is closed, and limits to the desired one.
\end{proof}

\begin{proposition}\label{prop:limit_tevelev_generic}
%Consider a limit Tevelev point $(V,[f_0:f_1])\in\Coll^1_d(\cC/B)_0$ and its aspect $V_{\spine}$ on $C_{\spine}$. 
Suppose $(V,[f_0:f_1])\in\Coll^1_d(\cC/B)_0$ satisfies properties (1), (2), (2') of Proposition \ref{prop:limit_tev_constraints}. Then, $(V_{\spine},[f_0:f_1])$ satisfies the following genericity properties.
\begin{enumerate}
    \item[(3)] $f_0,f_1$ share no common zeroes on $C_{\spine}$. In particular, $f_0,f_1$ are linearly independent, and $V_{\spine}=\langle f_0,f_1\rangle$ has no base-points.
    \item[(4)] Properties (1) and (2) of Proposition \ref{prop:limit_tev_constraints} hold ``exactly,'' that is, $V_{\spine}$ is simply ramified at each $s_1,\ldots,s_g$, and $[f_0:f_1]$ has ramification index exactly $r_h$ at $t_{h}$ for $h=1,2,\ldots,k$.
\end{enumerate}
\end{proposition}

\begin{proof}
We divide the proof into two steps. 
\begin{enumerate}[label=\underline{Step \arabic*}]
\item We first show that $f_0,f_1$ are linearly independent. Suppose instead that $[f_0:f_1]=[\lambda_0f:\lambda_1f]$. Note that we require $f\in V_{\spine}$ and that $V_{\spine}$ is ramified at $s_1,\ldots,s_g$.

Define the integer $\alpha\ge0$ as follows. We declare $\alpha=0$ if $q_h\neq[\lambda_0:\lambda_1]$ for all $h$, $\alpha=r_{h}$ if $q_h=[\lambda_0:\lambda_1]$.

Then, for every $q_h$ for which $q_h\neq[\lambda_0:\lambda_1]$, condition (2) implies that $f$ must vanish at $t_{h}$ to order $r_{h}$. In total, $f$ is constrained to vanish on a fixed divisor $D_{\alpha}\subset C_{sp}$ of degree $n-\alpha$.

On the other hand, suppose that $q_h=[\lambda_0:\lambda_1]$. Then, we have the additional condition from (2') that $V_{\spine}$ is ramified to order $r_{h}$ at $t_{h}$.

We now claim that linear series $V_{\spine}$ with all of the above properties cannot exist. To see this, we consider a further degeneration of our marked $\bP^1$ in which all of the points in the support of $D_{\alpha}$ coalesce to a single point $p_{\alpha}$, and consider the limit $\widetilde{V}_{\spine}$ of $V_{\spine}$ on this new marked rational curve. Then, we still have the conditions that $\widetilde{V}_{\spine}$ is ramified at $s_1,\ldots,s_g$, and that it is ramified to order $r_{h}$ at $t_{h}$. The condition that we have a non-zero section $f$ vanishing on $D_{\alpha}$ now becomes the condition that $\widetilde{V}_{\spine}$ is ramified to order $n-\alpha-1$ at $p_{\alpha}$.

The total amount of ramification on $V_{\spine}$ is therefore at least
\begin{equation*}
    g+(n-\alpha-1)+(\alpha-1)=g+n-2=2d-1,
\end{equation*}
which is more than the amount allowed by the Riemann-Hurwitz formula, a contradiction.
\item Now we prove that conditions (3) and (4) hold.
 \\Let $\mathcal{H}$ be the Hurwitz stack parametrizing maps $ \mathbb{P}^1 \rightarrow \mathbb{P}^1 $ of the same degree and with exactly the same ramification profiles as $[f_0:f_1]:\mathbb{P}^1 \rightarrow \mathbb{P}^1$. In the domain curve, we mark all the points $p_i,s_j,t_h$. On the target curve,  we mark $q_h$ for $h=1,...,k$ unless both $f_0$ and $f_1$ vanish at $t_h$ to order $r_h$.
  %\\For $\pi \in \mathcal{H}$, the total number of marked or branched points in the target $\mathbb{P}^1$ is $\leq n-r_{\tot}+k+g+(n-r_{\tot}+k-3)$ where $n-r_{\tot}+k-3$ is the maximum number of ramification points in the target $\mathbb{P}^1$ not in $\{p_i\} \cup \{t_i\} \cup \{s_i\}$. It follows that the dimension of $\mathcal{H}$ is 
 %\begin{equation} \label{dimension H0}
 %\mathrm{dim}(\mathcal{H})\leq g+2(n-r_{\tot}+k)-6.
 %\end{equation}
 %On the other hand, 
 
 We have a forgetful map $$\tau:\mathcal{H} \to M_{0,g+k}\times M_{0,N}$$ remembering the marked source and target. Because all of our points have been chosen to be general, $\tau$ is dominant and thus $\mathrm{dim}(\mathcal{H}) \geq \mathrm{dim}(M_{0,g+k}\times M_{0,N})$. A parameter count shows that this is only possible when $[f_0:f_1]$ has degree $d$ (i.e. $f_0,f_1$ share no common factor), $[f_0:f_1]$ has ramification index exactly $r_h$ at $t_{h}$ for $h=1,2,\ldots,k$, and has ramification index exactly $2$ at $s_i$ for $i=1,...,g$. The proof is complete.
\end{enumerate}
\end{proof}

\begin{remark}
Given (3), condition (2') is superfluous, as the linear independence of $f_0,f_1$ implies that the section
\begin{equation*}
        q^1_{h}\cdot f_0-q^0_{h}\cdot f_1
\end{equation*}
from (2) is not identically zero.
\end{remark}

\begin{proposition}\label{prop:limit_tevelev_tails}
Consider a limit Tevelev point $(V,[f_0:f_1])\in\Coll^1_d(\cC/B)_0$. Then, the aspects of $V$ on the tails of $C_{\spine}$ have the following properties.
\begin{enumerate}
    \item[(5)] The aspect $V_{E_j}$ is the unique linear series on $E_j$ with vanishing sequence $(d-2,d)$ at $s_j$ for $i=1,2,\ldots,g$.
    \item[(6)] The aspect $V_{R_h}$ is the unique linear series on $R_h\cong\bP^1$ spanned by a section $f^h_0$ vanishing to order $d$ at $t_h$ and a section $f^h_1$ vanishing to order $d-r_h$ at $t_h$ and to order 1 at each of
\begin{equation*}
p_{n-r_{h}-r_{h+1}-\cdots-r_{k}+1},\ldots,p_{n-r_{h+1}-\cdots-r_{k}}
\end{equation*}
\end{enumerate}
In particular, $V$ is a fine limit linear series, and the data of a limit Tevelev point is determined by $(V_{\spine},[f_0:f_1])$.
\end{proposition}

\begin{proof}
By properties (1) and (4), the vanishing sequence of $V_{E_j}$ at $s_j$ is at least $(d-2,d)$, which is the largest possible, and uniquely determines $V_{E_j}$ (as the image of the complete linear system $|\cO_{E_j}(2s_j)|$ in $|\cO_{E_j}(ds_j)|$. This yields (5).

By properties (2) and (4), the vanishing sequence of $V_{R_h}$ at $t_h$ is at least $(d-r_h,d)$, so $V_{R_h}$ contains a section vanishing to order $d$ at $t_h$. On the other hand, $V_{R_h}$ contains a non-zero section vanishing at the $r_h$ marked points
\begin{equation*}
p_{n-r_{h}-r_{h+1}-\cdots-r_{k}+1},\ldots,p_{n-r_{h+1}-\cdots-r_{k}}
\end{equation*}
because the same is true on the generic fiber. This section can then vanish to order at most $d-r_h$ at $t_h$, so we must have equality, yielding (6).

The final claims are immediate.
\end{proof}

We now show that the necessary conditions $(1)-(6)$ for a point of $\Coll^1_d(\C/B)_0$ to be a limit Tevelev point are sufficient.

\begin{proposition}\label{prop:lls_smoothing}
Suppose that $(V,[f_0:f_1])\in \Coll^1_d(\cC/B)_0$ satisfies conditions $(1)-(6)$ above. Then, $(V,[f_0:f_1])$ is a limit Tevelev point.

Moreover, $(V,[f_0:f_1])$ smooths transversally to the generic fiber, that is, the scheme-theoretic closure of the locus of Tevelev points is \'{e}tale over $B$. 
\end{proposition}

\begin{proof}
% We argue as in \cite[Theorems 3.3 and 3.4]{EH}. Namely, we cut out the locus of Tevelev and limit Tevelev points in $\Coll^1_d(\cC/B)$ inside a flat ambient scheme over $B$ and show that, at each point satisfying (1)-(7) on the special fiber, the intersection is reduced of the expected dimension 0. Thus, by generic smoothness and semicontinuity, these points smooth uniquely to the general fiber, and furthermore their smoothings lie in the locus where $f_0,f_1$ share no common factor, because the same is true on the special fiber. In particular, the points on the special fiber satisfying (1)-(7) are limit Tevelev.

% To carry this out, we employ the same construction as in \cite[Theorem 3.3]{EH}. Note that we may restrict our attention to fine limit linear series...
Let $G^1_d(\cC/B)$ be the moduli space of \textit{fine} limit linear series on the fibers of $\cC$, as constructed in the proof of \cite[Theorem 3.3]{EH}. 
%with vanishing sequences at the nodes $s_j,t_h$ satisfying the conditions imposed by properties (1), (3'), (6), (7) \textcolor{red}{now I think we want only (1) and (6)?}. 
By \cite[Theorem 4.5]{EH}, the special fiber of $G^1_d(\cC/B)$ is pure of the expected relative dimension $2d-2-g$, so $G^1_d(\cC/B)\to B$ is flat of the same dimension. Indeed, $G^1_d(\cC/B)$ may be regarded as a locally closed subscheme inside a product of Grassmannian bundles over $B$, which has the expected dimension over the special fiber, and thus must be flat of the expected dimension globally by semicontinuity.

As before, let $\cV_{\spine}$ be the rank 2 vector bundle on $G^1_d(\cC/B)$ with fiber equal to the linear series $V$ on the general fiber of $\psi:\cC\to B$, and fiber $V_{\spine}$ on special fiber. Let $\Coll^1_d(\cC/B)$ be the $\bP^3$-bundle over $G^1_d(\cC/B)$ associated to the rank 4 bundle $\cV_{\spine}^{\oplus 2}$.

We now cut out the locus of Tevelev points $\Coll^1_d(\cC/B)_{\Tev}\subset \Coll^1_d(\cC/B)$ via the $n$ linear conditions that $[f_0:f_1]$ is incident to $q_h$ at $p_i$ when $n-\sum_{h=j}^{k}r_h+1\le i\le n-\sum_{h=j+1}^{k}r_h$ and $1\le j\le k$. These conditions have flat limits given by condition (2) of Proposition \ref{prop:limit_tev_constraints}.

We see now that the points on the special fiber of $\Coll^1_d(\cC/B)_{\Tev}$ satisfying properties (1)-(6) therefore lie in the closed subscheme cut out by these linear conditions. Moreover, in a neighborhood of these points, the dimension of this closed subscheme, when restricted to the special fiber, has the expected dimension of 0, by a dimension count as in the proof of Proposition \ref{prop:limit_tevelev_generic}. Therefore, applying semicontinuity once more implies that the points satisfying (1)-(6) smooth in a flat family of relative dimension 0 to Tevelev points on the general fiber.

It remains to check the transversality. Again, we appeal to the dictionary between linear series and Hurwitz spaces: a non-zero relative tangent vector to $\Coll^1_d(\cC/B)_{\Tev}$ over $B$ would give rise to a non-zero relative tangent vector of the map $\tau:\cH\to M_{0,g+k}\times M_{0,k}$ where $\cH$ is the Hurwitz space of covers $C_{\spine}=\bP^1\to\bP^1$ defined by the points of $\Coll^1_d(\cC/B)_{\Tev}$, and $\tau$ remembers the marked source and target. However, we have chosen a general point of $\tau:\cH\to M_{0,g+k}\times M_{0,k}$, over which $\tau$ is unramified by generic smoothness, so we have reached a contradiction. This completes the proof.
\end{proof}

\subsection{Enumeration of limit Tevelev points}\label{lim_tev_enum}

By the results of the previous section, we have reduced to the following problem: fix 
\begin{equation*}
    (\bP^1,s_1,\ldots,s_g,t_1,\ldots,t_k)\in M_{0,g+k}
\end{equation*}
general. Then, we wish to count $(V,[f_0:f_1])\in\Coll^1_d(\bP^1)$ satisfying the following properties (cf. Propositions \ref{prop:limit_tev_constraints}, \ref{prop:limit_tevelev_generic} and \ref{prop:lls_smoothing}): 
\begin{enumerate}
    \item[(1)] $V$ is ramified at each $s_j$ for $j=1,2,\ldots,g$;
    %\item[(2)] $[f_0:f_1]$ is incident at $q_i$ when evaluated at $p_i$ for $i=1,2,\ldots,n-r_{\tot}$;
    \item[(2)] $[f_0:f_1]$ is incident to order $r_h$ at $q_{h}$ when evaluated at $t_h$ for $h=1,2,\ldots,k$;
    \item[(3)] $f_0,f_1$ are linearly independent, that is, $(V,[f_0:f_1])$ lies away from the rank 1 locus.
\end{enumerate}
Note that in this degenerate problem we have dropped the subscript $_{\spine}$.
\begin{remark}
From the results of the previous section, given (1) and (2), asking for $f_0$ and $f_1$ to be linearly independent is the same as asking for $(2')$ and thus the same as asking for $f_0$ and $f_1$ to share no common zero.
\end{remark}
The space $\Coll^1_d(\bP^1)$ is the $\bP^3$-bundle over the Grassmannian $\Gr(2,d+1)$ associated to the vector bundle $\cV\oplus \cV$, where $\cV$ is the tautological rank 2 subbundle on $\Gr(2,d+1)$. As before, we denote the projection map by $\varphi:\Coll^1_d(\bP^1)=\bP(\cV\oplus\cV)\to\Gr(2,d+1)$.

Conditions (1), (2), respectively, give rise to the following cycles on $\Coll^1_d(\bP^1)$:

\begin{enumerate}
    \item[(i)] An intersection of $g$ Schubert cycles of class $\sigma_1\in A^1(\Gr(2,d+1))$, pulled back by $\varphi$;
 %   \item[(ii)] an intersection of $n-r_{\tot}$ relative hyperplanes of class $c_1(\cO_{\bP(\cV\oplus\cV)}(1))\in A^1(\Coll^1_d(\bP^1))$;
    \item[(ii)] an intersection of $\sum_{h=1}^{k}r_h=n$ relative hyperplanes of class $c_1(\cO_{\bP(\cV\oplus\cV)}(1))\in A^1(\Coll^1_d(\bP^1))$.
\end{enumerate}

In particular, the cycles (i), (ii) define an intersection in $\Coll^1_d(\bP^1)$ of dimension $(2(d-1)+3)-n-g=0$.

However, we note that the set-theoretic intersection of these cycles typically includes excess components supported in the rank 1 locus, and thus fail (3). Indeed, one may check via parameter counts that there exist points in this intersection where $[f_0:f_1]=[\lambda_0f:\lambda_1f]$ with $[\lambda_0:\lambda_1]=q_{h}$, and furthermore that if $r_h>2$, we have excess loci of positive dimension.

In order to avoid this excess intersection, we replace (2) with the following combination of properties (2) and (2') appearing in Proposition \ref{prop:limit_tev_constraints}.
\begin{enumerate}
    \item[(2'')] $V$ has a non-zero section vanishing to order at least $r_h$ at $t_{h}$ \textit{and} $[f_0:f_1]$ is incident at $q_{h}$ when evaluated at $t_h$ for $h=1,2,\ldots,k$.
\end{enumerate}

For each $h$, property (2'') defines a cycle (ii'') on $\Coll^1_d(\bP^1)$, again of codimension $r_h-1$, given by the intersection of (the pullback of) a Schubert cycle of class $\sigma_{r_h-1}\in A^{r_h-1}(\Gr(2,d+1))$ and a relative hyperplane on $\Coll^1_d(\bP^1)$. The following now shows that the new intersection avoids the rank 1 locus, and the excess intersections are zero-dimensional and transverse.

\begin{proposition}\label{prop:new_intersection_zero_dim}
    The intersection of the cycles (i), (ii''), is transverse and is supported away from the rank 1 locus.
    
    Moreover, the only base-points of $(V,[f_0:f_1])$ occur among the points $t_1,\ldots,t_k$. For each base-point $t_h$, we must have $r_h>1$, and $V$ has vanishing sequence exactly $(1,r_h)$.
    
    Otherwise, $(V,[f_0:f_1])$ satisfies properties (1),(2) exactly.
\end{proposition}

%\begin{proof}
%The claim that the intersection occurs away from the rank 1 locus is essentially the same as that of Proposition \ref{prop:limit_tevelev_generic}.

%The remaining set-theoretic claims follow similarly from dimension counts. The new feature here is the possibility of base-points among the $t_h$: here, the additional condition that $V$ has a base-point at $t_h$ decreases the expected number of moduli by 1, but the presence of the base-point makes the single incidence condition at $t_h$ redundant, so we obtain additional intersections in dimension 0. If, however, the vanishing sequence at $t_h$ is strictly greater than $(1,r_h)$, then the expected number of moduli becomes negative, and we obtain no contributions.

%On the other hand, if $V$ has a base-point away from the $t_h$, we find again that the expected number of moduli becomes negative: for example, at one of $p_1,\ldots,p_{n-r_{\tot}}$, where $V$ is not already constrained to be ramified, then the expected number of moduli drops by 2, which fails to be offset of the loss of the codimension 1 incidence condition.

%Finally, we obtain the transversality as in the proof of Proposition \ref{prop:lls_smoothing}: a non-zero tangent vector in the intersection gives rise to a non-zero relative tangent vector of the generically \'{e}tale forgetful morphism $\tau:\cH\to M_{0,n_1}\times M_{0,n_2}$ where $\cH$ is a Hurwitz space of covers $\bP^1\to\bP^1$, yielding a contradiction.
%\end{proof}

\begin{proof}
The proof is similar to that of Proposition \ref{prop:limit_tevelev_generic}. We divide it into three steps.
\begin{enumerate}[label=\underline{Step \arabic*}]
%The claim that the intersection occurs away from the rank 1 locus is a dimension count on $\mathrm{Gr}(2,d+1)$. Explicitly, if $f_0$ and $f_1$ are linearly dependent, then we have:
%\begin{enumerate}[label=(\alph*)]
    %\item $V$ satisfies condition $(1)$,
    %\item $V$ has a non-zero section vanishing %to order $r_h$ at $t_h$ for each %$h=1,...,k$, and
    %\item $V$ has a non-zero section vanishing %simultaneously at least at $n-r_{\tot}+k-1$ %of the points %$p_1,....,p_{n-r_{\tot}},t_1,...,t_k$. 
%\end{enumerate}
%Clearly, (a) defines $g$ cycles of codimension 1, (b) defines $k$ cycles of codimension $r_h-1$ for $h=1,...,k$, and (c) defines a cycle of codimension at least $n-r_{\mathrm{tot}}+k-2$. Their intersection, for generic points $s_j,t_h,p_i$ and $q_l$, has expected codimension
%\begin{align*}
%g+r_{\tot}-k+k-1+n-r_{\tot}-1&=g+n-1\\
%&>g+n-2\\
%&=\dim\Gr(2,d+1),
%\end{align*}
%so must be empty (for example, one can argue that that $V$ has too much total ramification upon passing to the limit when the points $p_1,....,p_{n-r_{\tot}},t_1,...,t_k$ coalesce).
\item We first show that if $(V,[f_0:f_1])$ lies in the intersection of (i) and (ii''), then $f_0$ and $f_1$ must be linearly independent. 
Suppose instead $f_0$ and $f_1$ are linearly dependent, and write $[f_0:f_1]=[\lambda f: \mu f]$. We have:
\begin{enumerate}[label=(\alph*)]
    \item $V$ satisfies condition $(1)$,
    \item $V$ has a non-zero section $u_h$ vanishing to order $r_h$ at $t_h$ for each $h=1,...,k$, and
    \item $V$ has a non-zero section (namely, $f$) vanishing simultaneously at all or all but one of the points $t_1,...,t_k$ (if $[\lambda:\mu]$ is equal to one of the $q_i$, then $f$ need not vanish at the corresponding point of $C_{\spine}$). 
\end{enumerate}

Let $J \subseteq \{1,...,k\}$ be the set of indices $h$ such that the sections $u_h$ and $f$ are linearly dependent for a general choice of $s_j,t_h$. Consider a degeneration of our marked $\mathbb{P}^1$ in which all of the points $t_h$ for $h \in J$ coalesce to a single point $p$, and consider the limit $\widetilde{V}$ of $V$ on the new marked curve. Then (a), (b) and (c) give:
\begin{enumerate}[label=(\alph*)]
    \item $\widetilde{V}$ satisfies condition $(1)$,
    \item $\widetilde{V}$ has a non-zero section $\widetilde{u}_h$ vanishing to order $r_h$ at $t_h$ for each $h \notin J$, and
    \item $\widetilde{V}$ has a non-zero section $\widetilde{f}$ vanishing at $p$ at least to order $-1+\sum_{h \in J}r_h$ and vanishing at $t_h$ for $h \notin J$.
\end{enumerate}
Moreover, we may assume that $\widetilde{u}_h$ and $\widetilde{f}$ are linearly independent for each $h \notin J$ (otherwise we also coalesce $p$ and $t_h$ and keep repeating this process). Then, the total ramification of $\widetilde{V}$ is at least 
$$\left(g+\sum_{h \notin J} r_h\right) + \left(-1+\sum_{h \in J} r_h\right) -1=g+n-2
$$
while, by the Riemann-Hurwitz formula, the maximum allowed is $g+n-3$. This is a contradiction.

\item Now, we prove the claim about base-points, again by a parameter count on Hurwitz spaces. As in Proposition \ref{prop:limit_tevelev_generic}, we find that demanding additional $f_0,f_1$ share a common factor imposes too many conditions on the cover $[f_0:f_1]$, with the following exception. If $f_0,f_1$ are required to vanish to order 1 on some $t_h$ for which $r_h>1$, then we have added exactly one vanishing condition, because $V$ was already constrained to be ramified at $t_h$. On the other hand, we lose the incidence condition at $t_h$, which is now automatic, so such $[f_0:f_1]$ are still expected to exist. On the other hand, if $r_h=1$, requiring that $f_0,f_1$ vanish at $t_h$ adds 2 conditions, so such $[f_0:f_1]$ are not expected to exist in this case.

More precisely, let $\mathcal{H}$ be the Hurwitz stack parametrizing maps $ \mathbb{P}^1 \rightarrow \mathbb{P}^1 $ of the same degree and with exactly the same ramification profiles as $[f_0:f_1]:\mathbb{P}^1 \rightarrow \mathbb{P}^1$. In the domain curve, we mark all the points $s_j,t_h$ while, this time, on the target curve we mark $q_{h}$ for $h=1,...,k$ unless both $f_0$ and $f_1$ vanish at $t_h$. For the corresponding forgetful map 
$$\tau:\mathcal{H} \to M_{0,g+k}\times M_{0,N
}$$ 
(remembering the marked source and target) to be dominant, a parameter count shows that the only possible common zeros of $f_0$ and $f_1$ are $t_1,...,t_k$, and that these zeroes must appear with multiplicity at most 1. This proves the second claim.
 
 \item Transversality is obtained as in the proof of Proposition \ref{prop:lls_smoothing}: a non-zero tangent vector in the intersection gives rise to a non-zero relative tangent vector of a generically \'{e}tale forgetful morphism $\tau:\cH\to M_{0,n_1}\times M_{0,n_2}$ where $\cH$ is a Hurwitz space of covers $\bP^1\to\bP^1$, yielding a contradiction.
\end{enumerate}

\end{proof}

\begin{proof}[Proof of Theorem \ref{main_thm_schubert}]
By Proposition \ref{prop:new_intersection_zero_dim} we need to count the points in the intersection of (i), (ii'') where $V$ is base point free.

First, the intersection number of the cycles (i), (ii'') is given by
\begin{align*}
    &\int_{\Coll^1_d(\bP^1)}\varphi^{*}\left(\sigma_1^g\cdot\prod_{h=1}^{k}\sigma_{r_h-1}\right)\cdot c_1(\cO(1))^{k}\\
    =&\int_{\Gr(2,d+1)}\sigma_1^g\cdot\left(\prod_{h=1}^{k}\sigma_{r_h-1}\right)\cdot\varphi_{*}(c_1(\cO(1))^{k})\\
    =&\int_{\Gr(2,d+1)}\sigma_1^g\cdot\left(\prod_{h=1}^{k}\sigma_{r_h-1}\right)\cdot s_{k-3}(\cV\oplus\cV)\\
    =&\int_{\Gr(2,d+1)}\sigma_1^g\cdot\left(\prod_{h=1}^{k}\sigma_{r_h-1}\right)\cdot\left(\sum_{i+j=k-3}\sigma_i\sigma_j\right)\
\end{align*}
Moreover, the intersection is transverse, and given set-theoretically by the points $(V,[f_0:f_1])$ described in Proposition \ref{prop:new_intersection_zero_dim}.

More generally, suppose $J\subset\{1,2,\ldots,k\}$ and consider the points $(V,[f_0:f_1])$ in the intersection of (i),(ii'') for which $V$ has a simple base-point at $t_h$ for all $h\in J$. Then twisting down $V$ by these $r_h$ and dividing $f_0,f_1$ by the polynomials vanishing on the $t_h$ gives a new point $(V_J,[f^J_0:f^J_1])\in\Coll^1_{d-\# J}(\bP^1)$.

We have that $(V_J,[f^J_0:f^J_1])$ satisfies property (1) as before, and satisfies property (2'') at the points $t_h$ with $h\in I:=\{1,2,\ldots,k\} \backslash J$. Now, at points $t_h$ with $h\in J$, $V_J$ is base-point free of ramification index (at least) $r_h-1$ (with no incidence condition). Computing the number of such $(V_J,[f^J_0:f^J_1])$ similarly to before yields
\begin{equation*}
    \int_{\Gr(2,d+1-\#J)}\sigma_1^g\cdot\left(\prod_{h\in I}\sigma_{r_h-1}\right)\cdot\left(\prod_{h\in J}\sigma_{r_h-2}\right)\cdot\left(\sum_{i+j=k-3-\#J}\sigma_i\sigma_j\right).
\end{equation*}
Note that this count includes $(V_J,[f^J_0:f^J_1])$ that may have simple base-points at $t_h$ with $h\in I$. Note that it is still the case that $f^J_0,f^J_1$ are linearly independent and that the intersection is transverse. We require $d-\#J>0$ and $k-3-\#J\ge0$; if this is not the case, the integral is interpreted to be zero.

Now, the limit Tevelev points are exactly those where $V$ is base-point free. Thus, varying over all possible $J\subset\{1,2,\ldots,k\}$ and applying inclusion-exclusion yields Theorem \ref{main_thm_schubert}.
\end{proof}

\subsection{Comparison with Propositions \ref{prop recursion} and \ref{prop for g=0}}\label{sec:recover_recursion}

As a check, we now demonstrate how Theorem \ref{main_thm_schubert} recovers the formulas of before, namely Propositions \ref{prop recursion} and \ref{prop for g=0}. We continue to assume for notational convenience that $\mu_h=(1)^{r_h}$.

\begin{proposition}
    Let $g \geq 1$, $\ell \in \mathbb{Z}$ an integer and $r_1,...,r_k \in \mathbb{Z}_{\geq 1}$ positive integers. Assume that the Equations \eqref{dimensional constraint}, \eqref{conditions on d}  and \eqref{condition on k} hold with 
    $
    \mu_h=(1)^{r_h}
    $
    for $h=1,...,k$. Then,
    the formula for $\Tev_{g,\ell,r_1,\ldots,r_k}$ of Theorem \ref{main_thm_schubert} satisfies the recursion of Proposition \ref{prop recursion}.
\end{proposition}

\begin{proof}
%Recall that we put ourself under the assumpsion that either $r_h>1$ for all $h=1,...,k$ or $k=0$. 
%In both cases, we can assume $r_1,\ldots,r_{k-1}>1$; otherwise, we can simply drop the indices where $r_h=1$.

We claim that the desired recursion already holds on each individual term of the formula of Theorem \ref{main_thm_schubert}, indexed by a partition $I\coprod J=\{1,2,\ldots,k\}$. Indeed, this is immediate from the Pieri rule
\begin{equation*}
    \sigma_{r_k-1}\cdot\sigma_1=\sigma_{r_k-2}\cdot\sigma_{11}+\sigma_{r_{k}}
\end{equation*}
and the fact that
\begin{equation*}
    \int_{\Gr(2,N)}\beta\cdot\sigma_{11}=\int_{\Gr(2,N-1)}\beta
\end{equation*}
for any $\beta\in A_{0}(\Gr(2,N-1))$.

%When instead $k=0$, the statement is an immediate equality.

%we match the summand of $\Tev_{g,\ell,r_1,\ldots,r_{k-1}}$ (where we allow $k=1$) indexed by $I\coprod J=\{1,2,\ldots,k-1\}$ with that indexed by the same partition in $\Tev_{g-1,\ell,r_1,\ldots,r_{k-1}}$ and the those indexed by the partitions $(I\cup\{k\})\coprod J$ and $I\coprod (J\cup\{k\})$ of $\Tev_{g-1,\ell,r_1,\ldots,r_{k-1},2}$. The content is then (writing $m=2d-3-r_{\tot}+k-\# J$)
%\begin{align*}
%    &\sigma_1^g\sum_{i+j=m}\sigma_i\sigma_j\\
%    =&\sigma_{11}\sigma_1^{g-1}\sum_{i+j=m-1}\sigma_i\sigma_j+\left(\sigma_1\sigma_1^{g-1}\sum_{i+j=m}\sigma_i\sigma_j-\sigma_{11}\sigma_1^{g-1}\sum_{i+j=m-1}\sigma_i\sigma_j\right),
%\end{align*}
%which is clear.
\end{proof}

\begin{proposition}
    Let $\ell \geq 0$ and $r_1,...,r_k     \in \mathbb{Z}_{\geq 1}$ positive integers such that the inequalities \eqref{conditions on d},     \eqref{dimensional constraint} and \eqref{condition on k} hold with 
    $
    \mu_h=(1)^{r_h}
    $
    for $h=1,...,k$. Then
    the formula for $\Tev_{0,\ell,r_1,\ldots,r_k}$ of Theorem \ref{main_thm_schubert} satisfies Proposition \ref{prop for g=0}, that is, we have $\Tev_{0,\ell,r_1,\ldots,r_k}=1$.
\end{proposition}

\begin{proof}
Fix a non-negative integer $m$, and define $\Tev^m_{0,\ell,r_1,\ldots,r_k}$ by 
\begin{align*}
    \sum_{I\coprod J=\{1,2,\ldots,k\}}(-1)^{\# J}\cdot\int_{\Gr(2,d+1-\#J)}&\left(\prod_{h\in I}\sigma_{r_h-1}\prod_{h\in J}\sigma_{r_h-2}\right)\\
    \cdot&\left(\sum_{\substack{i+j=k-3-\#J\\ i\ge m}}\sigma_i\sigma_j\right).
\end{align*}
where $\Tev^m_{0,\ell,r_1,\ldots,r_k}$ differs from $\Tev_{0,\ell,r_1,\ldots,r_k}$ in that we require $i\ge m$ in the last term.

We prove the stronger statement that $\Tev^m_{0,\ell,r_1,\ldots,r_k}=1$ whenever $0\le m\le d-1$ by induction on $k$. Note that when $m>d-1$, all of the terms vanish automatically, because $\sigma_{i}=0$ for $i>d-1$.

First, consider the base case $k=3$. Then, the statement is simply that
\begin{equation*}
    \int_{\Gr(2,d+1)}\sigma_{r_1-1}\sigma_{r_2-1}\sigma_{r_3-1}=1
\end{equation*}
whenever $r_1,r_2,r_3\le d$, which follows from the Pieri rule.

For the inductive step, fix a partition $I'\coprod J'=\{1,2,\ldots,k-1\}$, and consider the summands of $\Tev^m_{0,\ell,r_1,\ldots,r_k}$ indexed by $(I'\cup\{k\})\coprod J'$ and $I'\coprod (J'\cup\{k\})$. 

Matching the summand indexed by $(i+1,j)$ from the partition $(I'\cup\{k\})\coprod J'$ with that indexed by $(i,j)$ from the partition $I'\coprod (J'\cup\{k\})$ leaves one additional term from the first partition, coming from $i=m,j=k-4-\#J'-m$. By the Pieri rule, observe further that
\begin{equation*}
\sigma_{r_k-1}\sigma_{i+1}-\sigma_{r_k-2}\sigma_{i}\sigma_{11}=\sigma_{r_k+i}.
\end{equation*}
Thus, we find the two partitions in question contribute
\begin{align*}
    (-1)^{\# J'}\cdot&\int_{\Gr(2,d+1-\#J')}\left(\prod_{h\in I'}\sigma_{r_h-1}\prod_{h\in J'}\sigma_{r_h-2}\right)\cdot\left(\sum_{\substack{i+j=k-4-\#J'\\ i\ge m}}\sigma_{r_k+i}\sigma_j\right)\\
    +(-1)^{\# J'}\cdot&\int_{\Gr(2,d+1-\#J')}\left(\prod_{h\in I'}\sigma_{r_h-1}\prod_{h\in J'}\sigma_{r_h-2}\right)\cdot\sigma_{r_k-1}\cdot(\sigma_{m}\sigma_{k-3-\#J'-m})
\end{align*}
Summing over all $I',J'$ and reindexing the sum in the first term, we find that
\begin{align*}
\Tev^m_{0,\ell,r_1,\ldots,r_k}&=\Tev^{m+r_k}_{0,\ell,r_1,\ldots,r_{k-1}}+\sum_{I'\coprod J'=\{1,2,\ldots,k-1\}}(-1)^{\# J'}\\
\cdot&\int_{\Gr(2,d+1-\#J')}\left(\prod_{h\in I'}\sigma_{r_h-1}\prod_{h\in J'}\sigma_{r_h-2}\right)\cdot\sigma_{r_k-1}\cdot(\sigma_{m}\sigma_{k-3-\#J'-m}).
\end{align*}

By the inductive hypothesis, if $m+r_k\le d-1$, then $\Tev^{m+r_k}_{0,\ell,r_1,\ldots,r_{k-1}}=1$. On the other hand, if $m+r_k>d-1$, then every term in the sum defining $\Tev^{m+r_k}_{0,\ell,r_1,\ldots,r_{k-1}}$ vanishes, so $\Tev^{m+r_k}_{0,\ell,r_1,\ldots,r_{k-1}}=0$.

It therefore suffices to show that the second term is equal to 0 if $m+r_k\le d-1$ and 1 if $m+r_k>d-1$. We claim that the sum is equal to simply
\begin{equation*}
    \int_{\Gr(2,d+1)}\sigma_{r_k-1}\sigma_m\sigma_{(r_1+r_2+\cdots+r_{k-1})-2-m}=\int_{\Gr(2,d+1)}\sigma_{r_k-1}\sigma_m\sigma_{2d-1-r_k-m}.
\end{equation*}

To see this, we employ the same inductive strategy as before: we fix a partition $I''\coprod J''=\{1,2,\ldots,k-2\}$ and match the terms of the sum where $J'=J''$ and $J'=J''\cup\{k-1\}$, using the fact that
\begin{align*}
    &\sigma_{r_{k-1}-1}\sigma_{k-3-\#J''-m}-\sigma_{r_{k-1}-2}\sigma_{k-3-(\#J''+1)-m}\sigma_{11}\\
    =&\sigma_{r_{k-1}+(k-4)-\#J''-m}.
\end{align*}
Continuing in this fashion yields the claim.

Again, by the Pieri rule, the final integral is equal to 1 whenever each of the three terms does not vanish. Because $2\le r_k\le d$ and $0\le m\le d-1$, this happens exactly when $r_k+m\ge d$, which is exactly what we need.
\end{proof}


\begin{thebibliography}{99}

%\bibitem{acv} D. Abramovich, A. Corti, and A. Vistoli, {\em Twisted bundles and admissible covers}, \href{https://www.tandfonline.com/doi/abs/10.1081/AGB-120022434} {Commun. Algebra \textbf{8} (2003), 3547-3618}

\bibitem{bp} A. Buch and R.  Pandharipande, \emph{Tevelev degrees in Gromov-Witten theory}, arXiv 2112.14824

\bibitem{Cast} G. Castelnuovo, {\em Numero delle involuzioni razionali giacenti sopra una curva di dato genere}, \href{http://emeroteca.braidense.it/beic_attacc/sfoglia_articolo.php?IDTestata=927&CodScheda=00AI&IDT=34&IDV=388&IDF=0&IDA=19791}{Rendiconti R. Accad. Lincei {\bf 5} (1889), 130--133.}

\bibitem{Cela} A. Cela, {\em Quantum Euler class and virtual Tevelev degrees of Fano complete intersections
}, \href{https://arxiv.org/abs/2204.01151}{arXiv:2204.01151}

\bibitem{CPS} A. Cela, R. Pandharipande and J. Schmitt, {\em Tevelev degrees and Hurwitz moduli spaces}, \href{https://www.cambridge.org/core/journals/mathematical-proceedings-of-the-cambridge-philosophical-society/article/abs/tevelev-degrees-and-hurwitz-moduli-spaces/ED064F6460BA33FEB50F5207BE31C03B}{Math. Proc. Cambridge Philos. Soc. {\bf 173} (2022), 479-510.}

\bibitem{EH} D. Eisenbud and J. Harris, David Eisenbud and Joe Harris, \emph{Limit linear series: Basic theory}, \href{https://link.springer.com/article/10.1007/BF01389094}{Invent. Math. \textbf{85} (1986), 337--371.}

\bibitem{HM} J. Harris and D. Mumford, {\em On the Kodaira Dimension of the Moduli Space of Curves}, \href{https://link.springer.com/article/10.1007/BF01393371}{Invent. Math. {\bf 67} (1982),  23-86.}

\bibitem{FarkasLian} G. Farkas and C. Lian,
{\em Linear series on general curves
with prescribed incidence conditions}, \href{ https://www.cambridge.org/core/journals/journal-of-the-institute-of-mathematics-of-jussieu/article/linear-series-on-general-curves-with-prescribed-incidence-conditions/CA1BA12EF2131B247F0419A3F280F76D}{J. Inst. Math. Jussieu, to appear.}

\bibitem{Lian} C. Lian, {\em Non-tautological Hurwitz cycles}, 
 \href{https://link.springer.com/article/10.1007/s00209-021-02903-7}{Math. Z. {\bf 301} (2022), 173-198.}

\bibitem{lp} C. Lian and R. Pandharipande, \emph{Enumerativity of virtual Tevelev degrees}, \href{https://arxiv.org/abs/2110.05520}{arXiv 2110.05520.}

\bibitem{lieblichosserman} M. Lieblich and B. Osserman, {\em Universal limit linear series and descent of moduli spaces}, \href{https://link.springer.com/article/10.1007/s00229-018-1049-5}{Manuscripta Math. \textbf{159} (2019), 13-38.}

%\bibitem{mtv} E. Mukhin, V. Tarasov, and A. Varchenko, \emph{Schubert calculus and representations of the general linear group}, J. Amer. Math. Soc. \textbf{22} (2009), 909-940.

\bibitem{SvZ} J. Schmitt and J. van Zelm, {\em Intersections of loci of admissible covers with tautological classes}, \href{https://link.springer.com/article/10.1007/s00029-020-00603-4}{Selecta Math. {\bf 26} (2020).}

%\bibitem{stacks-project} The Stacks Project Authors, {\em Stacks Project}, \href{https://stacks.math.columbia.edu}{https://stacks.math.columbia.edu}, 2021.

\bibitem{tevelev} J. Tevelev, {\em Scattering amplitudes of stable curves}, \href{https://arxiv.org/abs/2007.03831}{arXiv 2007.03831.}

\bibitem{vainsencher} I. Vainsencher, {\em Complete collineations and blowing up determinantal ideals}, \href{https://link.springer.com/article/10.1007/BF01456098}{Math. Ann. \textbf{267} (1984), 417-432.}



\end{thebibliography}
\end{document}